\documentclass[ejs]{imsart}

\RequirePackage[OT1]{fontenc}
\RequirePackage{amsthm,amsmath,amssymb}
\RequirePackage[numbers]{natbib}
\RequirePackage[colorlinks,citecolor=blue,urlcolor=blue]{hyperref}
\usepackage{graphicx}

\doi{10.1214/154957804100000000}
\pubyear{0000}
\volume{0}
\firstpage{0}
\lastpage{0}
\startlocaldefs
\numberwithin{equation}{section}
\theoremstyle{plain}
\newtheorem{theorem}{Theorem}[section]
\newtheorem{definition}{Definition}[section]
\newtheorem{example}{Example}[section]
\newtheorem{remark}{Remark}[section]
\newtheorem{corollary}{Corollary}[section]
\newtheorem{lemma}{Lemma}[section]
\newtheorem{proposition}{Proposition}[section]
\endlocaldefs

\begin{document}

\begin{frontmatter}

% "Title of the Paper"
\title{Inequalities for the false discovery rate (FDR) under dependence
% \thanksref{t1}
}
% \thankstext{t1}{This is an original survey paper}
\runtitle{Inequalities for the FDR under dependence}

% indicate corresponding author with \corref{}
% \author{\fnms{John} \snm{Smith}\thanksref{t2}\corref{}\ead[label=e1]{smith@foo.com}\ead[label=e2,url]{www.foo.com}}
% \thankstext{t2}{Thanks to somebody} 
% \address{line 1\\ line 2\\ \printead{e1}\\ \printead{e2}}

\author{\fnms{Philipp} \snm{Heesen}\ead[label=e1]{heesen@math.uni-duesseldorf.de}}
% \address{\printead{e1}}
\and
\author{\fnms{Arnold} \snm{Janssen}\ead[label=e2]{janssena@math.uni-duesseldorf.de}}
% \address{\printead{e2}}
\address{Heinrich-Heine University D\"usseldorf\\
Universit\"atsstr. 1, 40225 D\"usseldorf, Germany
\printead{e1,e2}}

\runauthor{P. Heesen and A. Janssen}

\begin{abstract}
Inequalities are key tools to prove FDR control of a multiple test. The present paper studies upper and lower bounds 
for the FDR under various dependence structures of $p$-values, namely independence, reverse martingale dependence 
and positive regression dependence on the subset (PRDS) of true null hypotheses. The inequalities are based on exact 
finite sample formulas which are also of interest for independent uniformly distributed $p$-values under the null. As 
applications the asymptotic worst case FDR of step up and step down tests coming from an non-decreasing rejection curve is 
established. In addition, new step up tests are established and necessary conditions for the FDR control are discussed. 
The reverse martingale models yield sharper FDR results than the PRDS models. Already in certain multivariate normal 
dependence models the familywise error rate of the Benjamini Hochberg step up test can be different from the desired level $\alpha$. 
The second part of the paper is devoted to adaptive step up tests under dependence. The well-known Storey estimator is 
modified so that the corresponding step up test has finite sample control for various block wise dependent $p$-values. 
These results may be applied to dependent genome data. Within each chromosome the $p$-values may be reverse martingale 
dependent while the chromosomes are independent. 
\end{abstract}


\footnotetext{Supported by the Deutsche Forschungsgemeinschaft (DFG).}
% \tnotetext[Test]{Partially supported by the Deutsche Forschungsgemeinschaft (DFG)}

%\begin{keyword}
%\kwd{}
%\kwd{}
%\end{keyword}

% history:
% \received{\smonth{1} \syear{0000}}

%\tableofcontents

\end{frontmatter}

\section{Introduction}

High dimensional testing problems given by $n$ hypotheses and corresponding ordered $p$-values $p_{1:n}\leq \ldots \leq 
p_{n:n}$ of the $p$-value vector $(p_1,\ldots,p_n)$ are frequently judged by multiple tests, like step up and step down 
tests. These tests rely on the component wise comparison of the ordered $p$-values with a family of critical values 
$(\alpha_{i:n})_{i\leq n}$, see [{\color{blue}1-3}, {\color{blue}7-11},{\color{blue}17} {\color{blue}23-25}] for instance. 
% \cite{benjamini_hochberg,benjamini_yekutieli,benjamini_krieger_yekutieli,finner_gontscharuk_dickhaus,finner_dickhaus_roters,
% finner_roters,gavrilov_benjamini_sarkar,sarkar,storey,storey_taylor_siegmund,zeisel} 
The overall control of the error probability of first kind is often too restrictive and leads to very conservative multiple 
tests. Therefore, Benjamini and Hochberg \cite{benjamini_hochberg} promoted the false discovery rate (FDR) as error measure 
to control. The FDR is the expected ratio of the number of falsely rejected null hypotheses among the total number of rejections. 

Starting with the famous choice of critical values $\alpha_{i:n}=\frac{i}{n}\alpha$ by Benjamini and Hochberg 
\cite{benjamini_hochberg}, quite a lot of authors studied finite sample or asymptotic FDR control (by some given 
level $0<\alpha<1$) under various assumptions. Roughly speaking the finite sample research can be derived in two 
categories. When the critical values are deterministic, then different sufficient conditions and dependence 
concepts for the $p$-values were established in order to ensure FDR control at level $\alpha$, i.e. $FDR\leq\alpha$, 
see Benjamini and Hochberg \cite{benjamini_hochberg}, Benjamini and Yekutieli \cite{benjamini_yekutieli}, 
Blanchard and Roquain \cite{blanchard_roquain_2} and Finner et al. \cite{finner_dickhaus_roters} among others. 
In case of data dependent critical values $\hat \alpha_{i:n}$, which lead to adaptive multiple tests, typically 
the i.i.d. structure of the $p$-values of true null hypotheses is assumed to achieve FDR control, see Storey 
et al. \cite{storey_taylor_siegmund} and Sarkar \cite{sarkar} for instance. They include an estimation of the 
number of true null hypotheses in the critical values in order to exhaust the predetermined FDR level better. 
Another branch is the asymptotic FDR control, where milder assumptions like weak dependency may be considered.  

In this paper we will again revisit FDR inequalities for step up and step down tests. The results depend 
on three dependence structures for the $p$-values, namely the most restrictive basic independence (BI) model, 
the reverse martingale model and the positive regression dependence on a subset (PRDS) model, 
respectively, which are introduced in Section \ref{Basics} beside the basic notation. Martingale arguments were used in 
Chapter 3 of the dissertation of Scheer \cite{scheer} for 
the comparison of the FDR and the expected number of false rejections. Reverse martingale models naturally 
show up for instance for measurements under restrictions or in multivariate extreme value theory, see Example 
\ref{ExampleRevMartingale} which include a Marshall/Olkin type dependence structure. 
Section \ref{SecExamplesRevMart} gives some construction methods of reverse martingales and it is pointed out that the FDR of the 
classical Benjamini Hochberg step up test can exactly be calculated for the reverse 
martingale structure, whereas already strict inequalities hold under multivariate normal PRDS models, see Example \ref{NeuesExamplePosDep}. 
Section \ref{SecIneq} discusses FDR inequalities for all these models which include inequalities 
for the FDR and inequalities for the critical values of FDR controlling step up tests. New necessary and sufficient 
conditions for finite sample FDR control at level $\alpha$ are derived. In particular, critical values considered 
earlier by Finner et al. \cite{finner_dickhaus_roters} and Gavrilov et al. \cite{gavrilov_benjamini_sarkar} are 
discussed, see also Section \ref{SectionNewSUProcedures}. The inequalities can be used to modify the critical 
values of Gavrilov et al. \cite{gavrilov_benjamini_sarkar} for step up tests, confer Example \ref{ApplicationsExample002} 
for improved new tests. Theorem \ref{ApplicationsTheorem001} establishes an exact asymptotic formula for the worst case 
FDR of step up tests which come from an increasing rejection curve. Observe that our inequalities allow to treat the 
difficult case when the expected portion of true null hypotheses becomes maximal. That result can be compared with 
the asymptotic optimal rejection curve (AORC) of Finner et al. \cite{finner_dickhaus_roters} which compares concave 
rejection curves. For concave rejection curves, it is remarkable that the asymptotic worst case FDR value here is the 
same for the corresponding step up and step down tests, see Section \ref{NeuSec4-2}.

Section \ref{SecAdaptiveControl} deals with adaptive SU tests under dependence which has often been neglected in the past. 
The adaptive step up tests rely on conservative estimators $\hat n_0$ of the expected number of true null hypotheses. Mostly 
the basic independence model is assumed in the literature when the FDR of the adaptive test is shown to be controlled. 
We will point out that finite sample FDR control of adaptive step up tests is a difficult affair and can not be expected 
in general under dependence. Recall that the well-known Storey multiple test does not work under positive regression 
dependence on the subset of true null hypotheses, see Example \ref{Example6-1} for instance. We will give a simple condition 
which ensures asymptotic FDR control under different dependence structures, see Theorem \ref{AdaptiveControlTheorem1} and 
\ref{AdaptiveControlTheorem1Part2}. For fixed sequence of estimators these conditions may also be regarded as conditions 
for the possible dependence structures. 
Furthermore, finite sample control can be obtained for various adaptive step up tests under the reverse martingale model. 
Also necessary conditions for finite sample control will be presented. It is shown that under additional conditions 
some modified Storey estimators work for dependent but block wise independent $p$-values, see Theorem \ref{Theorem004}. 
Under the general assumptions these results are sharp and can not be improved, see Example \ref{Example6-1}. However, 
when all $p$-values are independent then the new blockwise test is conservative. 

Section \ref{Theoretical_results} contains exact technical FDR formulas which are used in our proofs. Some of them are 
of separate interest. Many  statements of this paper are applications of our central Lemma \ref{LemmaCentral}. Furthermore, 
all proofs are outlined in Section \ref{Theoretical_results}.

\section{Basic notation and dependence models}\label{Basics}

Throughout, we investigate models with different dependence structures. All of these models are based on the following 
basic model with random number of true null hypotheses. Let $(\Omega, \mathcal{A}, P)$ be a probability space and 
let 
\begin{equation}\label{model001}
 (\epsilon_i, U_i, \xi_i)_{i\leq n}: \Omega \longrightarrow (\{0,1\}\times [0,1]^2)^n
\end{equation}
be a multivariate random variable where $\epsilon_i=0$ codes the occurrence of a $p$-value $\xi_i$ of a false 
null hypothesis, for short false $p$-value, and $\epsilon_i=1$ the occurrence of a $p$-value $U_i$ of a true 
null hypothesis, for short true $p$-value, whose marginal distribution is the uniform distribution on $(0,1)$. 
Then the model of the $p$-values is given by 
% \label{IntroducedCase}
\begin{equation}\label{model002}
 p_i = \epsilon_i U_i +(1-\epsilon_i) \xi_i, \quad 1\leq i \leq n, \quad \mbox{and}\quad  N:=\sum_{i=1}^n \epsilon_i,
\end{equation}
where $N$ is the random number of true p-values. This model includes the well studied mixture model of Efron et al. 
\cite{efron_tibshirani_storey}, where $(U_i)_i, (\xi_i)_i$, and $(\epsilon_i)_i$ are i.i.d. and jointly independent. 
Observe that here $N$ is naturally random. 
Throughout, true or false null hypotheses are identified with their $p$-values and for short the corresponding 
$p$-values are called ``true'' or ``false'', respectively. Since our multiple tests only rely on $p$-values this 
identification may be justified. 
Moreover, we define $p=(p_1,\ldots,p_n)$ and $\epsilon=(\epsilon_1,\ldots, \epsilon_n)$. 
Below, further assumptions about the dependence structure of the vector (\ref{model001}) are introduced. To avoid trivial 
cases let $E(N)$ be always positive and let us assume that our observations are the order statistics of the $p$-values, which 
are introduced as 
\begin{equation}\label{oderstat001}
 p_{1:n} \leq p_{2:n} \leq \ldots \leq p_{n:n}.
\end{equation}
Moreover, let $\hat F_n(t) = \frac{1}{n}\sum_{i=1}^n 1\{p_i\leq t\}$, $0\leq t\leq 1$, be the empirical cumulative distribution 
function of the $p$-values.

Let $\mathcal{B}([0,1]^n)$ denote the Borel sets of $[0,1]^n$. A set $C \in \mathcal{B}([0,1]^n)$ is said to be decreasing 
iff, $c'\in C$ and $c\leq c'$ component-by-component imply $c\in C$.

\begin{definition}[Dependence structures]\label{DependenceStructures}
\hspace{1cm} \\
 (a) Let $(\epsilon_i, \xi_i)_{i\leq n}$, $U_1,\ldots, U_n$ be independent. Then we call this submodel for the $p$-values 
(\ref{model002}) to be the {\bf basic independence (BI) model}. Note that $(\epsilon_i, \xi_i)_{i\leq n}$ 
is considered as one random variable whereas $U_1,\ldots,U_n$ are considered as individual random variables in terms of 
independence. \\
(b) Let 
% the conditional probability given $p_i\leq t$ and $\epsilon = \bar \epsilon$ 
\begin{equation}\label{PRDS}
 t \mapsto P(p\in C | p_i \leq t , \epsilon = \bar\epsilon) 
\end{equation}
be non-increasing for every decreasing set $C \in \mathcal{B}([0,1])^n$, $\bar\epsilon \in \{0,1\}^n$ and all $i$ 
with $\epsilon_i=1$. Then  $p$-value model (\ref{model002}) is called the  {\bf PRDS} model (positive regression 
dependent on the subset of true null hypotheses). \\
% (c) If on the other hand (\ref{PRDS}) is non-decreasing we call it  {\bf NRDS} model (negative regression dependent 
% on a subset). \\
(c) Conditioned under $\epsilon$ let 
\begin{equation}\label{MartingaleDef}
 \frac{1\{p_i \leq t\}}{t},\quad 0< t \leq 1, \mbox{ for all } i \mbox{ with } \epsilon_i=1
\end{equation}
be a reverse martingale with respect to the reverse filtration 
$\mathcal{F}_t = \sigma((1\{p_j\leq s\},\epsilon_j) \, : \, 1\leq j \leq n, s\geq t)$. 
Then $p$-value model (\ref{model002}) is called {\bf reverse martingale} model.
\end{definition}

\begin{remark}\label{Remark2-1}
% wie stehen die modelle zueinander
(a) The assumptions for the PRDS model in Definition \ref{DependenceStructures} (b) are a little bit weaker than the usual 
PRDS assumptions, see Finner et al. \cite{finner_dickhaus_roters} for instance. In the literature it is sometimes called 
weak PRDS. Nevertheless, we will call it PRDS model for brevity. \\
(b) The BI model is a submodel of the PRDS and reverse martingale model. Furthermore, the intersection of the PRDS and 
reverse martingale models is at least greater than the BI model. To see this regard $k$ independent disjoint blocks of 
$(U_1,\ldots,U_n)$ with maximal dependence in each block given here by the same uniformly distributed random variable. 
% observe that the maximal dependent case for the true $p$-values with $U_1=\ldots = U_n$ is also 
% included in both models when $U_1$ is independent from $(\epsilon_i, \xi_i)_{i\leq n}$. \\
\end{remark}

We will see that the reverse martingale models yield sharper FDR results than the PRDS concept. A comparison of these models 
is included in Section \ref{SecExamplesRevMart}.

In the past literature usually conditional versions of the BI and PRDS model with deterministic 
$\epsilon_1,\ldots,\epsilon_n$ have been considered, see Benjamini and Hochberg \cite{benjamini_hochberg}, 
Benjamini and Yekutieli \cite{benjamini_yekutieli}, Blanchard and Roquain \cite{blanchard_roquain_2}, 
Finner et al. \cite{finner_dickhaus_roters} and Finner and Roters \cite{finner_roters} for instance. 
In all models defined above these conditional versions are included as special case.

\bigskip 
In this paper we mainly focus on step up tests (SU tests), which we briefly recall. Suppose that 
\begin{equation}\label{stepup001}
 0 < \alpha_{1:n} \leq \alpha_{2:n} \leq \ldots \leq \alpha_{n:n} < 1
\end{equation}
denote possibly data dependent critical values and set $\alpha_{0:n}=\alpha_{1:n}$ for convenience. The 
corresponding SU test is based on the number of rejections 
\begin{equation}\label{index001}
 R = \max\{ i \, : \, p_{i:n}\leq \alpha_{i:n}  \}
\end{equation}
and rejects the null hypotheses corresponding to the set of $p$-values 
$\{p_i \, : \,p_i \leq \alpha_{R:n}\}$. When the condition in (\ref{index001}) is empty no hypothesis is rejected 
and $R=0$ holds. Then equivalently all null hypotheses with $p$-values
\begin{equation}\label{rejected001}
 p_{i:n},\quad i\leq R 
\end{equation}
are rejected. Let 
\begin{equation}\label{rejected002}
 V=\#\{ i\, : \, p_i\leq \alpha_{R:n}, \epsilon_i = 1 \}
\end{equation}
be the unobservable number of falsely rejected true null hypotheses. The judgment of multiple tests is often done 
via the control of the celebrated ``false discovery rate'' (FDR) which is given by $FDR=E\left(\frac{V}{R}\right)$ 
(with $\frac{0}{0}=0)$. More generally we introduce the conditional false discovery rate
\begin{equation}\label{fdr0012}
 FDR(n_0):= E \left( \frac{V}{R} \Big| N=n_0 \right)
% , \quad FDR(n_0, \bar{f}):= E\left( \frac{V}{R} \Big| n_0, \bar{f} \right)
\end{equation}
as conditional expectation given $N=n_0$.
% , $n_0$ and the false $p$-values (\ref{model003}) respectively. 
% Observe that the models also include constant $N=n_0$.
%  and constant false $p$-values $\bar f$. 
The conditional quantity (\ref{fdr0012}) is a special case since constant $N$ are included. 
%  of the $FDR=E(\frac{V}{R})$. 

Benjamini and Hochberg \cite{benjamini_hochberg} promoted the FDR as new error criterion competing against the well 
known familywise error rate (FWER) and provided a multiple test procedure controlling the FDR under certain assumptions.
% Benjamini and Hochberg (1995) introduced the first multiple hypothesis test controlling the FDR. 
The so-called Benjamini Hochberg test (BH test) is the linear SU test with fixed critical values
\begin{equation}\label{BHcritval}
 \alpha_{i:n} = \frac{i}{n}\alpha, \quad 1\leq i \leq n.
\end{equation}
In our setting and notation it was shown by Benjamini and Yekutieli \cite{benjamini_yekutieli} and 
Finner and Roters \cite{finner_roters} that 
\begin{equation}\label{FDR001}
 FDR(n_0) = \frac{n_0}{n} \alpha
\end{equation}
holds for the conditional expectation in the BI model. Benjamini and Hochberg \cite{benjamini_hochberg} previously 
showed that ``$\leq$'' holds in (\ref{FDR001}) for the BI model. Moreover, Benjamini and Yekutieli \cite{benjamini_yekutieli} 
proved that ``$\leq$'' holds in (\ref{FDR001}) for the PRDS model.  
% Moreover, they showed 
In the same work, they also introduced a new SU test based on more conservative critical values
\begin{equation}\label{CritValBY}
 \alpha_{i:n}= \frac{i}{n \sum_{j=1}^n\frac{1}{j}}\alpha.
\end{equation} 
Benjamini and Yekutieli \cite{benjamini_yekutieli} pointed out that ``$\leq$'' again holds in (\ref{FDR001}) for this 
SU test under the basic model with arbitrary dependence structure of $(\epsilon_i, U_i, \xi_i)_{i\leq n}$. Blanchard 
and Roquain \cite{blanchard_roquain_2} showed that the critical values (\ref{CritValBY}) may also be replaced by 
\begin{equation}\label{CritValBR}
 \alpha_{i:n} := \frac{\alpha}{n}\int_0^i xd\nu(x),
\end{equation}
where $\nu$ is an arbitrary probability measure on $(0,\infty)$. For $\nu(\{i\})= (i\sum_{j=1}^n\frac{1}{j})^{-1}$, 
$i=1,\ldots, n$, the critical values correspond to (\ref{CritValBY}).
% Later Finner and Roters \cite{finner_roters} showed that the inequality 
% ``$\leq$`` in (\ref{FDR001}) is actually an equality ''=``.
Adaptive versions based on (\ref{CritValBY}) and (\ref{CritValBR}) are presented in Theorem \ref{AdaptiveControlTheorem1Part2} under arbitrary dependence.

% % % All proofs of the paper are outlined in section \ref{Theoretical_results}.

We will particularly focus on critical values coming from a continuous non-decreasing function 
\begin{equation}\label{ConcaveRejectionCurve}
 f:[0,1]\to [0,\infty) \mbox{ with } f(0)=0 \mbox{ and } f(x_0)=1
\end{equation}
for some $x_0 < 1$. We refer to $f$ as rejection curve. Moreover, let $f^{-1}$ denote 
the left continuous inverse of $f$ and let the deterministic critical values be generated via
\begin{equation}\label{critVal_rejection_curve_f}
 \alpha_{i:n} = f^{-1} \left( \frac{i}{n} \right), \quad 1 \leq i \leq n.
\end{equation}
We refer to $f^{-1}$ as critical value curve.
% Bsp rejection curve bei BH also Simes Gerade
Note that the BH test is based on the Simes line $f(t) = t/\alpha$, $t\in [0,1]$, see 
Finner et al. \cite{finner_dickhaus_roters} for instance.

Finner et al. \cite{finner_dickhaus_roters} introduced the Asymptotic Optimal Rejection Curve (AORC) which is constructed 
to have FDR control by $\alpha$ in an asymptotic Dirac uniform (DU) setting given by $\xi_i=0$, $i=1,\ldots, n$. The AORC is given by 
\begin{equation}\label{AORC_curve}
 f_\alpha(t) = \frac{t}{t(1-\alpha)+\alpha}, \quad t\in [0,1],
\end{equation}
but since $f(1)=\alpha_{n:n}=1$, the above assumptions for rejection curves for SU tests are not fulfilled. There are 
several modifications of the AORC and corresponding SU tests to overcome this problem. For further details we refer to 
Finner et al. \cite{finner_dickhaus_roters}.

\section{Examples of reverse martingale models and a comparison with PRDS}\label{SecExamplesRevMart}

At the beginning it is shown that there exist positive dependent multivariate normal models which are PRDS without 
the martingale property. To prove this we will consider the following example. 

\begin{example}\label{NeuesExamplePosDep}
 Let $X_1$ and $Y$ be i.i.d. standard normal random variables with distribution function $\Phi$. \newline 
(a) Consider the PRDS model 
$$ (X_1,X_2) = \left(X_1, \frac{1}{\sqrt{2}} X_1 +  \frac{1}{\sqrt{2}} Y\right)
$$
and related $p$-values $(p_1,p_2) = (\Phi(X_1),\Phi(X_2))$. Then the familywise error rate of the BH step up test 
with critical values (\ref{BHcritval}) at level $\alpha=0.5$ and $n=n_0=2$ is $FWER=FDR(2)=\frac{7}{16}$, cf. 
(\ref{FDR001}), i.e. less than $\alpha=0.5$. \newline 
(b) For the negative dependence model 
$$ (X_1,X_2) = \left(X_1, -\frac{1}{\sqrt{2}} X_1 +  \frac{1}{\sqrt{2}} Y\right)
$$
the familywise error rate of the BH step up test is $FWER=FDR(2)=\frac{9}{16}$ and hence greater than $\alpha=0.5$. 
\end{example}

The proof is given in Section \ref{Theoretical_results}. Note that Gavrilov et al. \cite{gavrilov_benjamini_sarkar}, 
p. 625, already derived Monte Carlo results showing that the FWER of the BH step up test may be strictly below $\alpha$ 
under PRDS. In contrast to PRDS the reverse martingale models allow sharper FDR results, see Section \ref{SecIneq} and 
Lemma \ref{LemmaCentral}. The next remark summarizes this. 

\begin{remark}
Proposition \ref{AppicationsLemmaDirect} always implies the formula $FDR(n_0)=\frac{n_0}{n}\alpha$, see (\ref{FDR001}), 
for the BH step up test under the reverse martingale model, whereas only ''$\leq$`` holds under PRDS. Since an strict 
inequality ''$<$`` shows up in the PRDS Example \ref{NeuesExamplePosDep} (a), that normal dependence model is no reverse 
martingale. 
\end{remark}

In conclusion we see that reverse martingale models allow sharper FDR formulas as under PRDS. However, we do not know whether every 
reverse martingale model is PRDS.

The reverse martingale structure yields a rich class of $p$-values. The next example shows how to construct new 
reverse martingale models from known ones. In particular, reverse martingale measures on product spaces of $[0,1]$ 
are preserved under a lot of operations.  

\begin{example}\label{ExampleBspIngredients}
Let $\epsilon$, $I\subset\{i \, : \, \epsilon_i=1\}$ and $J\subset\{i\, : \, \epsilon_i=0\}$ be fixed with $|I|>0$.
Define $p$-values via the canonical projections $p_i : [0,1]^{|I|+|J|} \to [0,1]$. Then 
\begin{equation*}
 \mathcal{M}_{I}(I \cup J) = \left\{ \begin{array}{l}
                                           P \mbox{ distribution on } [0,1]^{|I|+|J|} : \\
					   P \mbox{ represents a reverse martingale model with true}  \\
					   \mbox{p-values } p_i, \, i\in I \mbox{ and false p-values } p_i, \, i\in J
                                         \end{array}
 \right\}
\end{equation*}
is the set of reverse martingale measures. \\ 
(a) $\mathcal{M}_{I}(I\cup J)$ is closed under mixtures including convex combinations. \\
(b) (Independent coupling of reverse martingale regimes) Suppose that there are partitions $\{i  :  \epsilon_i=1\}=\sum_{i=1}^r I_i$ and $\{i  :  \epsilon_i=0\}=\sum_{i=1}^r J_i$ 
with $I_i\neq \emptyset$ for all $i$ but $J_i$ is allowed to be empty. Whenever $P_i \in \mathcal{M}_{I_i}(I_i \cup J_i)$ 
holds for all $i\leq r$, then the product measure $\bigotimes_{i=1}^r P_i$ belongs to $\mathcal{M}_{\{i : \epsilon_i=1\}}(1,\ldots,n)$. \\
% (c) Note that more generally (\ref{MartingaleDef}) is preserved under independent uniform permutations of the $p$-values. \\
(c) (Optional switching of reverse martingales) It is well known that two independent reverse martingale models 
given by $P_1$ and $P_2 \in \mathcal{M}_{\{i : \epsilon_i=1\}}(1,\ldots,n)$ may be combined as follows. Let $\tau:[0,1]^n\to(0,1]$ 
denote a reverse stopping time w.r.t. model $P_1$. Whenever $1_{[0,\tau]}(p_i)=0$ holds, that $p$-value for the index $i$ comes from 
the $P_1$ model. In case $1_{[0,\tau]}(p_i)=1$, consider the $p$-value $\tilde p_i$ of $P_2$ and take the renormalized value 
$\tilde p_i \tau$ as new $p$-value for these coordinates. \\
(d) Let $p_1,\ldots,p_n:\Omega\to[0,1]$ be any random variables such that 
\begin{equation}
 t \mapsto \frac{\frac{1}{n} \sum_{i=1}^{n}1_{[0,t]}(p_i)}{nt}
\end{equation}
is a reverse martingale. Suppose that $i\mapsto\sigma(i)$ denotes a uniformly distributed permutation of $\{1,\ldots,n\}$ jointly 
independent of the $p_i$'s. Then the family of $p$-values $(p_{\sigma(1)},\ldots,p_{\sigma(n)})$ has the reverse martingale 
property (\ref{MartingaleDef}) for each component. 
\end{example}

Part (d) of that example is easy to prove and grew out of a discussion with Julia Benditkis which is kindly acknowledged. 

The subsequent example gives an explicit example which may have the following practical meaning. The statistician can only observe 
a concentration $X_i$ above a joint random ground level $Y$. It also occurs in multivariate extreme value theory and risk analysis 
when the $Z_i$ have a joint risk component $Y$.

\begin{example}\label{ExampleRevMartingale}
(Marshall and Olkin type dependence, see Marshall and Olkin \cite{nelson}) Let $X_1,\ldots,X_n,Y$ denote continuous, independent, 
real random variables, where $X_1,\ldots,X_n$ are i.i.d.. Consider $Z_i := \max(X_i,Y)$ for $1\leq i \leq n$. 
% If $H^{-1}$ is the inverse distribution function of $H(t)=P(Z_1\leq t)$, 
% then $p_i:=H(Z_i)$ 
The transformed true $p$-values $p_i := H(Z_i)$, $i=1,\ldots,n$ given by $H(t)=P(Z_1\leq t)$ have the reverse martingale property, 
see Section \ref{Theoretical_results} for a proof. It is easy to verify, that the present model is also PRDS since it is 
based on a comonotone transformation of the i.i.d. model $X_1,\ldots,X_n,Y$. Notice that in case $n=2$ the negative variables 
$(-Z_1,-Z_2)$ correspond to the bivariate Marshall and Olkin \cite{nelson} model. 
\end{example}

% The model of Example \ref{ExampleRevMartingale} may have the following practical meaning. The statistician can only observe a 
% concentration $X_i$ above a joint random ground level $Y$. It also occurs in multivariate extreme value theory and risk analysis 
% when the $Z_i$ have a joint risk component $Y$. 
% % Further research about reverse martingale models will be forthcoming. 

Reverse martingale models can also be obtained allowing some dependence between ``true and false'' $p$-values. 
% We do not like to overload this paper and we will add 
% further material about martingale assumptions in forthcoming papers. A possible reverse martingale structure is given in the 
% next example for instance. \\
% (d) If (\ref{PRDS}) is non-increasing, i.e. if the $p$-values are negative regression dependent on the subset of 
% true null hypotheses, then some of the subsequent inequalities turn around, see Heesen \cite{heesen} for instance. 

\begin{example}
(Dependence between the null and alternatives) The following models may be used as ingredients for 
Example \ref{ExampleBspIngredients} (b). Consider a distribution $P_1$ on $[0,1]^{|I_1|+|J_1|}$, where 
\begin{itemize}
 \item[(i)] the marginal distribution of the ``trues``  $(p_i)_{i\in I_1}$ belongs to $\mathcal{M}_{I_1}(I_1)$.  
 \item[(ii)] Suppose $\min(p_i : i\in I_1) \geq \max(p_i : i\in J_1)$ holds $P_1$ almost everywhere. 
\end{itemize}
Then $P_1 \in \mathcal{M}_{I_1}(I_1\cup J_1)$ holds. The proof follows the same line as in Example \ref{ExampleRevMartingale}. \\
\end{example}

\section{Inequalities for the FDR}\label{SecIneq}

In this section new inequalities are derived which are used in the proceeding chapters. 
We start with arbitrary non-decreasing deterministic critical values $0< \alpha_{1:n}\leq \ldots \leq \alpha_{n:n}<1$ 
and the following question. 
\begin{itemize}
 \item[\textbullet] What can be said about the FDR of the corresponding SU test given by a fixed model (\ref{model002})? 
\end{itemize}

The next inequalities rely on more technical results given in Section \ref{Theoretical_results}, in particular in 
Lemma \ref{LemmaCentral}. 
% All proofs can be found in section \ref{Theoretical_results}. Together with the technical Lemma \ref{LemmaCentral} 
% of section 6 the present Lemma \ref{ApplicationsLemma001} establishes lower and upper FDR bounds.

\begin{proposition}\label{AppicationsLemmaDirect}
(a) Assume the reverse martingale model (including the BI model) and consider the SU test with arbitrary deterministic 
critical values (\ref{stepup001}). Then we have 
 \begin{equation}\label{lower_and_upper_FDR_bound}
 \frac{E(N)}{n} \left( \min_{i \leq n} \frac{n\alpha_{i:n}}{i} \right) \leq FDR \leq 
\frac{E(N)}{n} \left( \max_{i \leq n} \frac{n\alpha_{i:n}}{i} \right).
\end{equation}
(b) Suppose that $P(R=j)>0$ holds. 
\begin{itemize}
 \item[(i)] The inequality $\frac{n\alpha_{j:n}}{j} < \max_{i\leq n} \left( \frac{n\alpha_{i:n}}{i} \right)$ implies the strict inequality
\end{itemize}
\begin{equation}
 FDR < \frac{E(N)}{n} \left( \max_{i \leq n} \frac{n\alpha_{i:n}}{i} \right).
\end{equation}
\begin{itemize}
 \item[(ii)] Conversely, $\frac{n\alpha_{j:n}}{j} > \min_{i\leq n} \left( \frac{n\alpha_{i:n}}{i} \right)$ implies
\end{itemize}
\begin{equation}
  \frac{E(N)}{n} \left( \min_{i \leq n} \frac{n\alpha_{i:n}}{i} \right) < FDR.
\end{equation}
(c) Under the PRDS model we still obtain  
\begin{equation}\label{upper_FDR_bound}
FDR \leq \frac{E(N)}{n} \left( \max_{i \leq n} \frac{n\alpha_{i:n}}{i} \right). 
\end{equation}
% and under NRDS
% \begin{equation}\label{lower_FDR_bound}
% FDR \geq \frac{E(N)}{n} \left( \min_{i \leq n} \frac{n\alpha_{i:n}}{i} \right).
% \end{equation}
\end{proposition}

Observe that Example \ref{NeuesExamplePosDep} gives a counterexample that the lower bound in (\ref{lower_and_upper_FDR_bound}) 
does not hold under PRDS. 

With different methods Guo and Rao \cite{guo_rao} already showed that the upper bound in (\ref{lower_and_upper_FDR_bound}) holds
under the PRDS property. Moreover, Sarkar \cite{sarkar2002} derived several inequalities and exact expressions for the FDR for 
so-called generalized step-up-down tests. These inequalities are then used as key tools to prove FDR control of an step-up-down 
test basically with Benjamini Hochberg critical values (\ref{BHcritval}) under the PRDS assumption and a further step-down test 
under multivariate total positivity of order 2 (MTP$_2$). 

Under regularity assumptions the inequalities are asymptotically sharp. We refer to Section \ref{Theoretical_results} 
and Lemma \ref{ApplicationsLemma003}.

For deterministic critical values let us discuss the assumption 
\begin{equation}\label{SecApplications002}
 j\mapsto \frac{\alpha_{j:n}}{j} \quad \mbox{is non-decreasing.}
\end{equation}
It is easy to verify that (\ref{SecApplications002}) holds for the critical values (\ref{critVal_rejection_curve_f}) 
which come from a concave rejection curve. 
% Consider deterministic critical values (\ref{critVal_rejection_curve_f}) coming from a concave rejection 
% curve $f$. For these critical values it is easy to verify that 
% \begin{equation}\label{SecApplications002}
%  j\mapsto \frac{\alpha_{j:n}}{j} \quad \mbox{is non-decreasing.}
% \end{equation}
Under (\ref{SecApplications002}) Benjamini and Yekutieli \cite{benjamini_yekutieli} showed that Dirac uniform (DU) configurations 
(i.e. $\xi_i=0$) are least favorable parameter 
configurations for the FDR in the BI model for fixed $N=n_0$. Let us assume that the critical values with 
(\ref{SecApplications002}) lead to overall finite sample FDR control for the BI model, the PRDS model or the 
martingale model, respectively. Then the subsequent results investigate necessary conditions for the critical values 
$\alpha_{i:n}$ itself and the following question can be treated. 
\begin{itemize}
 \item[\textbullet] What can be said about the critical values $\alpha_{i:n}$ when the FDR is controlled by $FDR \leq \alpha$ 
for all distributions given by a specified class of submodels for fixed $n$? 
\end{itemize}

% \medskip
\begin{lemma}\label{ApplicationsLemma002}
 Suppose that the SU test with deterministic critical values (\ref{stepup001}), satisfying (\ref{SecApplications002}), 
always controls the $FDR$ at level $\alpha$ (i.e. $FDR \leq \alpha$) under all distributions of the BI model.\\
(a) A necessary condition is then $\alpha_{j:n} \leq \frac{j\alpha}{n+1-j}$ for all $ 1\leq j \leq n$.\\
(b) Suppose that in addition to (\ref{SecApplications002}) we have $\frac{\alpha_{k:n}}{k} < \frac{\alpha_{k+1:n}}{k+1}$ for one $k<n$. 
Then we have $\alpha_{j:n} < \frac{j\alpha}{n+1-j}$ for all $j \leq k$.
\end{lemma}

\begin{corollary}\label{ApplicationsCorollary001}
Consider the assumptions of Lemma \ref{ApplicationsLemma002}. \\
(a) The inequality $\alpha_{1:n} \leq \frac{\alpha}{n}$ always holds.\\
(b) If $\alpha_{1:n}=\frac{\alpha}{n}$ then the SU test is already a BH test at level $\alpha$. 
Otherwise, $\alpha_{1:n} < \frac{\alpha}{n}$ holds.\\
(c) If $\alpha_{1:n}=\frac{\beta}{n}$ for some $\beta \leq \alpha$, then always $FDR \geq \frac{\beta E(N)}{n}$ follows.
\end{corollary}

\begin{remark} 
(a) Consider the Dirac uniform configuration DU$(n_0)$ with $N=n_0$ and $\xi_i=0$, $i=1,\ldots,n$, for the reverse martingale 
model. Under (\ref{SecApplications002}) the lower bound in (\ref{lower_and_upper_FDR_bound}) can then be improved following 
the lines of Proposition \ref{AppicationsLemmaDirect} for $DU(n_0)$ by 
\begin{equation}\label{Neue_lower_bound}
 n_0 \frac{\alpha_{n+1-n_0:n}}{n+1-n_0} \leq FDR_{DU(n_0)}.
\end{equation}
% The result follows by a reinspection of the proof of Proposition \ref{AppicationsLemmaDirect} since $R\geq n+1-n_0$ holds 
% on $\{ V>0\}$ for DU$(n_0)$ configurations and by (\ref{SecApplications002}).\\
% from (\ref{lower_and_upper_FDR_bound}) by considering new critical values $\max(\frac{i}{n+1-n_0}\alpha_{n+1-n_0:n},\alpha_{i:n})$.
(b) The statements of Lemma \ref{ApplicationsLemma002} and Corollary \ref{ApplicationsCorollary001} naturally hold if 
$FDR \leq \alpha$ holds for all szenarios described by the PRDS or reverse martingale model, since the BI model is a submodel 
of both models.
% The necessary conditions of Lemma \ref{ApplicationsLemma002} and Corollary \ref{ApplicationsCorollary001} already 
% follow from the BI models and they do not use the full PRDS or reverse martingale model. 
% %  Since the BI model is contained in the PRDS model and the reverse martingale model, Lemma \ref{ApplicationsLemma002} and 
% % Corollary \ref{ApplicationsCorollary001} (a), (b) also hold for those models. Moreover, 
% % Corollary \ref{ApplicationsCorollary001} (c) holds for the reverse martingale model. 
\end{remark}

The next example demonstrates an application of our inequalities.

\begin{example}[about necessary conditions for the BI model]
\hspace{1cm}\newline 
(a) SU tests with critical values 
\begin{equation}\label{CritValBeliebig}
 \alpha_{j:n}=\frac{j\alpha}{n+b-ja} \quad\mbox{and non-negative } a \mbox{ and } b
\end{equation}
are frequently 
discussed in the literature. The requirement $\alpha_{n:n}<1$ for all $n$ implies $0 \leq a \leq 1-\alpha$. A necessary 
condition for $FDR \leq \alpha$ is then by Lemma \ref{ApplicationsLemma002} the additional condition $a\leq b$. If $a>0$ 
is positive then $a<b$ is necessary. \\
(b) Consider some fixed integer $1 \leq k < n$ and the adjusted critical values 
\begin{equation}\label{adjusted_critical_values}
 \alpha_{j:n}=\frac{j\alpha}{n-j(1-\alpha),} \quad j\leq k < n, 
\end{equation}
of the Asymptotic Optimal Rejection Curve (AORC) (\ref{AORC_curve}) of Finner et al. \cite{finner_dickhaus_roters} 
which are first only specified for $j \leq k$. 
There are several possibilities for the choice of $\alpha_{j:n}$, $k<j\leq n$, for the extension of 
(\ref{adjusted_critical_values}) such that (\ref{SecApplications002}) remains true and $\alpha_{n:n}<1$ holds, see 
(\ref{ModificationCritVal}) below and confer also Finner et al. \cite{finner_dickhaus_roters} and Gontscharuk 
\cite{gontscharuk}. It is well-known by Finner et al. \cite{finner_dickhaus_roters} that the SU tests with adjusted 
critical values (\ref{adjusted_critical_values}) do not have finite sample FDR control but asymptotic FDR control. 
Since $a=1-\alpha$ and $b=0$ we directly observe by (a) that finite sample FDR control can not hold. Even the first 
critical value $\alpha_{1:n} = \frac{\alpha}{n-(1-\alpha)} \neq \frac{\alpha}{n}$ is too large to allow FDR control. 
\end{example}

\section{Applications under independence}
\subsection{FDR control}\label{NeuSec4-1}

Our inequalities include a device for the choice of adequate parameters $a,b$ for the critical values (\ref{CritValBeliebig}). 
Below, we restrict ourselves to the FDR adjustment under the BI model. Some technical inequalities presented in Section 
\ref{Theoretical_results} also work under dependence. 

\begin{proposition}\label{proposition_29Jan001}
Consider SU tests with critical values (\ref{CritValBeliebig}) for $0<a<b$ with fixed value $b$ and an adjustment of $a$.
Let $\mathcal{P}_{BI}$ stand for all distributions of the BI model and let $FDR_{(b,a)}$ be linked to (\ref{CritValBeliebig}). 
There exists a unique parameter $a_1 \in (0,b)$ with 
\begin{equation}\label{FDR_adjustment001}
 \sup_{P \in \mathcal{P}_{BI}} FDR_{(b,a_1)} = \alpha.
\end{equation}
The worst case $FDR_{(b,a)}$ is strictly smaller (larger) than $\alpha$ for $a<a_1$ ($a>a_1)$. 
\end{proposition}

Sharper inequalities for the range of the parameter $a_1$ of (\ref{FDR_adjustment001}) are included in 
Proposition \ref{PropositionSharperIneq} which may be of computational interest in practice. However, 
the exact value $a_1$ should be calculated by numerical calculations, see Lemma \ref{LemmaProp3-2ab} 
(a) of the proof section. For the step up test with critical values (\ref{CritValBeliebig}), $b=1$, 
$\alpha=0.05$ and $n=10$ we have $a_1\approx 0.92$ for instance which is not far away from the upper 
bound $b=1$.

% As long as the SU tests rely on deterministic critical values (\ref{stepup001}) the assumptions can be weakened. Then 
% the reverse martingale condition (\ref{MartingaleDef}) is only necessary for $t\in\{\alpha_{1:n},\ldots,\alpha_{n:n},1\}$. 
% Observe that then (\ref{MartingaleDef}) is always a reverse semimartingale and martingale arguments seem to be very natural. 

\medskip 
% \section{Improvements of SU tests with deterministic critical values}
\label{SectionNewSUProcedures}

%
% % % Ab hier Section 5
%

In the next step we establish another FDR adjustment as in (\ref{FDR_adjustment001}) of critical values which 
may have some advantage in practice. The new proposal relies on the following observation. 
Typically the largest coefficients of (\ref{stepup001}) are responsible for a worst case FDR value with 
$FDR > \alpha$, cf. Finner et al. \cite{finner_gontscharuk_dickhaus}. For these reasons we propose to bound the largest 
critical values as follows. 

\begin{proposition}\label{FurtherFDRAdjustmentProp}
 Fix $\epsilon>0$ which is typically small. Consider SU tests with deterministic critical values (\ref{stepup001}) 
satisfying (\ref{SecApplications002}) and $\alpha_{1:n} < \frac{\alpha}{n}$. Introduce for fixed $1\leq k \leq n$ 
the new coefficients 
\begin{equation}\label{ModificationCritVal}
 \alpha_{j:n}^{(k)} := \min \left( \alpha_{j:n}, \frac{j}{k}\alpha_{k:n} \right), \quad j=1,\ldots,n.
\end{equation}
If $\sup_{P\in \mathcal{P}_{BI}} FDR((\alpha_{j:n})_{j}) > \alpha + \epsilon$ holds for the FDR of the corresponding 
SU test, then there exists some $1\leq k_0 < n$ with 
\begin{equation}\label{FurtherFDRAdjustmentPropFormula}
 \sup_{P\in \mathcal{P}_{BI}} FDR((\alpha_{j:n}^{(k)})_{j}) \leq \sup_{P\in \mathcal{P}_{BI}} FDR((\alpha_{j:n}^{(k_0)})_{j}) \leq \alpha+\epsilon
\end{equation}
for all $k \leq k_0$ and ''$>$`` for all $k>k_0$.
\end{proposition}

The modification (\ref{ModificationCritVal}) of the critical values has also been considered by Finner et al. \cite{finner_dickhaus_roters} Example 3.2 for the 
special case of critical values coming from the AORC. Moreover, for this type of modification Finner et al. \cite{finner_gontscharuk_dickhaus} propose to increase 
the parameter $b$ in a further step in order to decrease the FDR below $\alpha$. In contrast to earlier work Proposition \ref{FurtherFDRAdjustmentProp} works for 
general critical values with (\ref{SecApplications002}). 
% In actual fact, $\epsilon=0$ is also possible in Proposition \ref{FurtherFDRAdjustmentProp} and $k=1$ leads 
% to an BH test with $FDR\leq\alpha$. But here we try to exhaust the predetermined FDR level $\alpha$ over a wide range of DU configurations without exceeding it 
% too much. Therefore, $\epsilon>0$ is needed. 
In principal, Proposition \ref{FurtherFDRAdjustmentProp} also works for $\epsilon=0$. For practical purposes a choice of a very small value $\epsilon>0$ can be 
recommended, see Figure \ref{fig:Bild1} in Example \ref{ApplicationsExample002}.

\begin{example}\label{ApplicationsExample002}
(Under the BI model)\\
(a) Let us consider the step down critical values 
\begin{equation}\label{crit_val_gavrilov_benjamini_sarkar}
 \alpha_{j:n} = \frac{j\alpha}{n+1-j(1-\alpha)}, \quad j\leq n, 
\end{equation}
of Gavrilov et al. \cite{gavrilov_benjamini_sarkar}. It is well-known that the corresponding SD test, see Section 
\ref{SectionApplSD} for the notation, yields 
finite sample FDR control, whereas the corresponding SU test has no finite sample FDR control. On the other hand 
the necessary conditions for finite sample FDR control of Lemma \ref{ApplicationsLemma002} (a) are fulfilled. 
In this case our results do not exclude this procedure but we get a meaningful lower bound based on (\ref{Neue_lower_bound}) 
for the worst case of FDR$(n_0)$ and a hint how the critical values (\ref{crit_val_gavrilov_benjamini_sarkar}) can be modified.\\
(b) Figure \ref{fig:Bild1} shows the FDR of the SU test for $\alpha=0.05$ and $n=300$ for the least favorable DU configurations 
for different values of $N=n_0$ with $n_0=1,\ldots, n$, given by the critical values (\ref{crit_val_gavrilov_benjamini_sarkar}). 
The lower bound in (\ref{lower_and_upper_FDR_bound}) is based on
\begin{equation}
 \min_{i\leq n} \frac{n\alpha_{i:n}}{i} = \frac{\alpha}{1+\frac{\alpha}{n}} \to \alpha.
\end{equation}
Thus, this lower bound is close to the FDR for fixed $N \approx n$ for this example. 
\begin{figure}[htbp]\label{Figure2}
% [htbp] 
  \centering
     \includegraphics[scale=0.55]{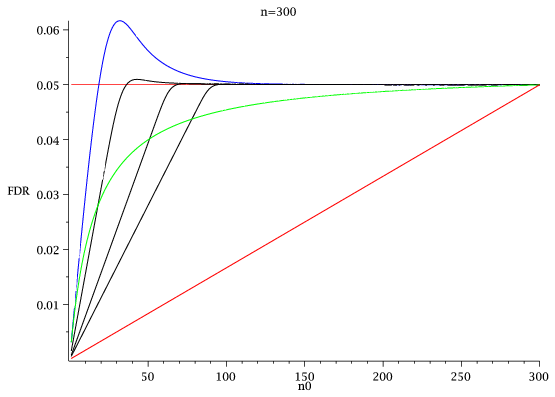}
  \caption{SU FDR of the Gavrilov et al. \cite{gavrilov_benjamini_sarkar} critical values (blue curve), critical values 
(\ref{ModificationCritVal}) for $k=283,250,223$ (black curves), BH SU FDR (lower red line), and lower bound (\ref{Neue_lower_bound}) (green curve)}
  \label{fig:Bild1}
\end{figure}
Moreover, Figure \ref{fig:Bild1} shows the FDR$_{DU}$ plot for different choices $k=300,283,250,$ $223$ given by 
(\ref{ModificationCritVal}) based on the critical values (\ref{crit_val_gavrilov_benjamini_sarkar}) with graphs decreasing in $k$. 
The straight line represents the FDR of the BH test and the green curve is the lower bound (\ref{Neue_lower_bound}).
Numerical results yield the value $k_0 = 283$ for $\epsilon=10^{-3}$ and $k_0 = 223$ for $\epsilon=10^{-4}$ given by 
(\ref{FurtherFDRAdjustmentPropFormula}), see Table \ref{table001}. Here, $k=1$ leads to a BH test and is the only $k$ 
with $FDR\leq\alpha$. \\
(c) In practice the statistician can accept the enlarged FDR value $\alpha+\epsilon$ or he can reduce the critical values 
(\ref{crit_val_gavrilov_benjamini_sarkar}) by a minor reduction of $\alpha$. 
\begin{table}[h]
  \centering
\begin{tabular}{c|ccccc}
 $k$ & 300 & 283 & 250 & 223 & 2 \\ \hline
 & & &  & &\\
$\sup_{P \in \mathcal{P}_{BI}}$FDR & 0.06165 & 0.05098 & 0.05020 & 0.05009 & 0.050006 \\ 
 & & & & & \\ \hline
 & & & & & \\
argmax$_{n_0}$ FDR$(n_0)$& 32 & 43 & 74 & 100 & 300\\
 & & & & &
\end{tabular}
  \caption{Worst case FDR for different choices of $k$ in (\ref{ModificationCritVal}) for the critical values 
(\ref{crit_val_gavrilov_benjamini_sarkar}).}
   \label{table001}
\end{table}
\end{example}

The results given in Figure \ref{fig:Bild1}  are quite promising. A minor modification of the critical values 
(\ref{crit_val_gavrilov_benjamini_sarkar}) exploits the FDR of the BH test. The value of FDR$(n_0)$ is quite 
good for large $n_0$, where the power of the multiple test is really needed.

\subsection{Asymptotic worst FDR case}\label{NeuSec4-2}

Our technique applies to the worst case FDR asymptotics for SU tests given by rejection curves $f$.

\begin{theorem}\label{ApplicationsTheorem001}
 Let $\mathcal{P}_n$ be the set of all possible distributions of the BI model for fixed $n$. Consider a non-decreasing 
continuous rejection curve $f:[0,x_0]\to[0,1]$ for some $0<x_0<1$ with $f(0)=0$, $f(x_0)=1$. Assume also that $f$ is 
left-sided differentiable on $(0,x_0)$ and let $f(x)\geq(1+\epsilon)x$ for all $x$ and some $\epsilon>0$. For the 
sequence of SU tests based on the critical values (\ref{critVal_rejection_curve_f}) the asymptotic worst SU case FDR is 
\begin{equation}\label{beta_ApplicationsTheorem002}
 \beta := \limsup_{n\to \infty} \sup_{P_n \in \mathcal{P}_n} FDR_{P_n} = \sup\left\{\frac{x}{1-x}\frac{1-f(x)}{f(x)} : 0 \leq x \leq x_0\right\}.
\end{equation}
Moreover, $0<\beta<1$ holds.
\end{theorem}
% Alt
% \begin{theorem}\label{ApplicationsTheorem001}
%  Let $\mathcal{P}_n$ be the set of all possible distributions of the BI model for fixed $n$. Consider the sequence of 
% SU tests given by a concave rejection curve $f$ satisfying (\ref{ConcaveRejectionCurve}) and critical values defined by 
% (\ref{critVal_rejection_curve_f}). The asymptotic worst SU case FDR 
% \begin{equation}\label{beta_ApplicationsTheorem001}
%  \beta := \limsup_{n\to \infty} \sup_{P_n \in \mathcal{P}_n} FDR_{P_n}
% \end{equation}
% fulfills $0 < \beta < 1$. It's value satisfies the fix point equation $H(1-\beta)=\beta$, where 
% \begin{equation}\label{beta_ApplicationsTheorem002}
%  H(t)= \sup\left\{ \frac{x}{1-x} \frac{1-f(x)}{f(x)} \, : \, 0<x \mbox{ and } \frac{f(x)-x}{1-x}\leq t \right\}
% \end{equation}
% for $0<t<1$.
% \end{theorem}

\begin{remark}
(a) Note that the AORC curve $f_\alpha$ (\ref{AORC_curve}) yields
\begin{equation}
 \frac{x}{1-x} \frac{1-f_{\alpha}(x)}{f_\alpha(x)} = \alpha, \quad x\in (0,1),
\end{equation}
which again supports the optimality of $f_\alpha$. \\
(b) Suppose that $\beta \leq \alpha < 1$ holds for some rejection curve $f$ treated in Theorem \ref{ApplicationsTheorem001}. The proof of Theorem 
\ref{ApplicationsTheorem001} implies $f\geq f_\alpha$ on $[0,x_0]$ with the upper bound $f^{-1}(\frac{i}{n}) \leq f^{-1}_\alpha(\frac{i}{n})$ for 
the critical values of the SU tests.  \\
% Then easy geometric arguments, given in the proof of Theorem \ref{ApplicationsTheorem001} already imply $f(x) \geq f_\alpha(x)$ 
% for all $x \in[0,x_0]$. Consequently, the critical values based on the rejection curve $f$ by (\ref{critVal_rejection_curve_f}) 
% are more conservative than the critical values based on the AORC. \\
(c) The question about the limiting FDR for $\frac{R_n}{n}\to 0$ was addressed in Remark 5.1 of 
Finner et al. \cite{finner_dickhaus_roters}. Our approach contributes to this open problem. 
\end{remark}

% \section{Asymptotic worst case step down FDR}
\label{SectionApplSD}
For concave rejection curves we briefly point out that the asymptotic bound $\beta$ of Theorem \ref{ApplicationsTheorem001} 
is the same for step down (SD) tests. Consider again the critical values (\ref{critVal_rejection_curve_f}). 
The step down critical index is given by the number of SD rejections
\begin{equation}\label{RSDdef}
 R_{SD} := \max\{ i\, : \, p_{j:n} \leq \alpha_{j:n} \mbox{ for all } 1\leq j \leq i \}.
\end{equation}
The modification of (\ref{index001}) and (\ref{rejected001}) for SD tests requires that all null hypotheses with $p$-values 
\begin{equation}
 p_i \leq  \alpha_{R_{SD}:n}
\end{equation}
are rejected. When the condition in (\ref{RSDdef}) is empty no hypothesis is rejected and $R_{SD}=0$ holds. 
Similarly to (\ref{rejected002}) put $V_{SD} = \#\{\mbox{true } p_i\, : \, p_i \leq \alpha_{R_{SD}:n} \}$
to be the number of false positive rejections. 

\begin{theorem}\label{ApplicationsTheoremSD}
 Under the assumptions of Theorem \ref{ApplicationsTheorem001}, let us additionally assume that $f$ is a concave rejection curve. 
Then we obtain 
% have for some $P_n \in \mathcal{P}_n$
% \begin{equation}
%  FDR_{P_n,SD} := E_{P_n}\left( \frac{V_{SD}}{R_{SD}} \right) \leq FDR_{P_n,SU}
% \end{equation}
% and 
the same asymptotic upper bound
\begin{equation}\label{ApplicationsTheoremSD001}
 \limsup_{n\to \infty} \sup_{P_n \in \mathcal{P}_n} FDR_{P_n,SD} = \beta
\end{equation}
for the sequence of SD tests generated by (\ref{critVal_rejection_curve_f}) as for the corresponding sequence of SU tests.
\end{theorem}

\begin{remark} 
 Consider a sequence of SD tests with critical values given by (\ref{critVal_rejection_curve_f}) based on a concave rejection curve 
(\ref{ConcaveRejectionCurve}) and which has finite sample FDR control by $\alpha$ for all $n \in \mathbb{N}$ in the BI model. Then 
the same holds asymptotically for the corresponding sequence of SU tests. This technique does not apply for the sequence of SU tests 
based on the Gavrilov et al. \cite{gavrilov_benjamini_sarkar} critical values (\ref{crit_val_gavrilov_benjamini_sarkar}), since the 
critical values are not generated by one rejection curve, but by the sequence of rejection curves $(1+\frac{1}{n})f_\alpha$ based on 
the AORC. 
% It is well known and also easy to see by Figure \ref{fig:Bild1} 
% that a SU test based on (\ref{crit_val_gavrilov_benjamini_sarkar}) does not control the FDR at finite sample size for all values 
% of $n_0$. 
% Observe that since the asymptotic worst case SD and SU FDR coincide, the corresponding sequence of SU tests automatically 
% has asymptotic FDR control by $\alpha$ in the BI model.
\end{remark}

\section{Applications to adaptive control under dependence}\label{SecAdaptiveControl}

In contrast to the preceding sections now data driven critical values are considered in order to exhaust the FDR level 
of given SU tests. Much effort was done in order to establish adaptive SU tests which are based on the linear SU test 
of Benjamini and Hochberg \cite{benjamini_hochberg}. These tests are typically based on conservatively biased estimators 
$\hat n_0$ of $N$ in order to exploit the FDR level better. The approach is motivated by the substitution of $\alpha$ by 
$\alpha'=\frac{n}{\hat n_0}\alpha$ which leads to the heuristic $FDR\approx\alpha E\left(\frac{N}{\hat n_0}\right)\approx 
\alpha$ for consistent $\hat n_0$ and to data dependent BH type critical values 
\begin{equation}\label{data-dependent-crit-val001}
 \hat \alpha_{i:n} = \frac{i}{\hat n_0} \alpha.
\end{equation} 
We refer to the well-known and frequently applied so called Storey type estimators given by the empirical distribution 
function $\hat F_n$ of the $p$-values
\begin{equation}\label{Storey_estimator_adapt_contr}
 \hat n_0(\lambda) = n \frac{1-\hat F_n(\lambda)+\kappa_n}{1-\lambda}, \quad \kappa_n>0,
\end{equation}
where $\lambda$ is often chosen to be close to $0.5$, see Storey et al.\cite{storey_taylor_siegmund} and Storey\cite{storey} 
for the choice of $\kappa_n=\frac{1}{n}$. 
There are several estimators and conditions for FDR control in the literature, for example see Benjamini et al. 
\cite{benjamini_krieger_yekutieli}, Sarkar \cite{sarkar} and Zeisel et al. \cite{zeisel}.

The finite sample FDR control of the adaptive SU test of Storey based on the critical values (\ref{data-dependent-crit-val001}) 
and estimator (\ref{Storey_estimator_adapt_contr}) with $\kappa_n=\frac{1}{n}$ seems to be restricted to the BI model. Even for 
the reverse martingale model, which allows that some $p$-values coincide, further assumptions are required, see Example 
\ref{Example6-1} below for instance. 

The aim of this section is twofold for the reverse martingale model. 
\begin{itemize}
 \item In Section \ref{SubsectionAsymptotic} sufficient conditions for estimators of $N$ are introduced which ensure asymptotic FDR control. 
 \item In Section \ref{SubsectionBlockModel} modified Storey SU tests are introduced which have finite sample FDR control for various block wise 
		      dependence models.
\end{itemize}
Moreover, we propose an adaptive multiple test for arbitrary dependent data and we also give a sufficient condition 
for the estimator and dependence structure, respectively, which again ensures asymptotic FDR control, compare with 
(\ref{CritValBY}) and (\ref{CritValBR}). 

The estimators and multiple tests are based on various assumptions. Let $0 < \lambda < 1$ and divide $[0,1]$ 
into two areas, the rejection area $[0,\lambda]$ and the estimation area $[\lambda,1]$. Let us specify different 
assumptions. 
\begin{itemize}
\item[(A1)] The unknown value $N$ is estimated by an estimator
\end{itemize}
	      \begin{equation}\label{53estimation}
		\hat n_0 = g((\hat F_n(t))_{\lambda\leq t\leq 1}) > 0 \quad \mbox{almost surely}
	      \end{equation}
\begin{itemize}
\item[]	   via the empirical cumulative distribution function $\hat F_n$ on the \textbf{estimation area} $[\lambda,1]$
and a measurable function $g$. 
%  \item[(A1)] The unknown value $N$ is estimated by (\ref{53estimation}) via the empirical cumulative distribution function 
% 	      $\hat F_n$ on the \textbf{estimation area} $[\lambda,1]$. Let 
% \end{itemize}
% 	      \begin{equation}\label{53estimation}
% 		\hat n_0 = g((\hat F_n(t))_{\lambda\leq t\leq 1}) > 0 \quad \mbox{almost surely}
% 	      \end{equation}
% \begin{itemize}
% \item[]	      be a positive estimator of $N$.
\item[(A2)]  The unknown value $N$ is estimated by
\end{itemize}
	      \begin{equation}
		\hat n_0 = g((\hat F_n(t))_{0\leq t\leq 1}) > 0 \quad \mbox{almost surely.}
	      \end{equation}
\begin{itemize}
\item[(A3)] The multiple test is applied to the \textbf{rejection area} $[0,\lambda]$ with data dependent critical 
	      values
\end{itemize}
	      \begin{equation}\label{CritVal002}
	       \hat \alpha_{i:n} = \left(\frac{i}{\hat n_0} \alpha\right)\wedge \lambda, \ 1\leq i\leq n.
% 		\quad	\mbox{where } a\wedge b := \min(a,b).
	      \end{equation}
\begin{itemize}
\item[(A4)] The multiple test is applied to the \textbf{rejection area} $[0,1]$ with the following data dependent 
	    critical values
\end{itemize}
	      \begin{equation}\label{AdaptiveControlTheorem1002}
	       \hat\alpha_{i:n} = \frac{\alpha}{n} \int_0^{in/\hat n_0}xd\nu(x), \quad 1\leq i\leq n,
	      \end{equation}
\begin{itemize}
\item[] where $\nu$ is an arbitrary probability measure on $(0,\infty)$.
\end{itemize}
Taking the minimum in (\ref{CritVal002}) goes back to Storey et al. \cite{storey_taylor_siegmund} and ensures 
that one does not reject $p$-values greater than $\lambda$, therefore the name rejection area. Statisticians 
often do not like to reject a hypothesis when the $p$-value is too high. The estimated critical values 
(\ref{AdaptiveControlTheorem1002}) are based on the deterministic family of critical values (\ref{CritValBR}) 
of Blanchard and Roquain \cite{blanchard_roquain_2} which also include the critical values (\ref{CritValBY}) 
of Benjamini and Yekutieli \cite{benjamini_yekutieli}.

\subsection{Asymptotic results}\label{SubsectionAsymptotic}
The central Lemma \ref{LemmaCentral} now establishes sufficient conditions for asymptotic FDR control of adaptive SU 
tests under different dependence structures.

\begin{theorem}\label{AdaptiveControlTheorem1}
Let $\mathcal{P}_n$ be the set of all possible distributions of the reverse martingale models for fixed $n$ and let 
$(P_n)_n$ be a sequence of distributions with $P_n \in \mathcal{P}_n$. Moreover, let $\hat n_{0,n}$ be a sequence of 
estimators for $N_n$ which fulfill (A1). 
%  $\hat n_{0,n} > 0$ almost surely and let each $\hat n_{0,n}$ be a function of 
% $(\hat F_n(t))_{t\geq \lambda}$ for some $0<\lambda\leq 1$ . 
If 
\begin{equation}\label{AdaptiveControlTheorem1001}
P_n\left( \frac{\hat n_{0,n}}{N_n} \leq 1-\delta \right) \longrightarrow 0
%  P_n\left(\hat\pi_{0,n} > \frac{E_{P_n}(N_n)}{n}-\epsilon\right) \longrightarrow 1 
\end{equation}
holds for all $\delta>0$, where $\frac{x}{0}:=\infty$ for $x>0$, then 
\begin{equation}\label{StatementAdaptiveControlTheorem1}
 \limsup_{n \to \infty} FDR_{P_n} \leq \alpha
\end{equation}
holds for the sequence of adaptive SU tests given by (A3).
% with critical values 
% \begin{equation}\label{AdaptiveControlTheorem10012}
%  \hat\alpha_{i:n} = \left( \frac{i}{\hat n_{0,n}}\alpha \right)\wedge \lambda.
% \end{equation}
\end{theorem}

Finner and Gontscharuk \cite{finner_gontscharuk} and Gontscharuk \cite{gontscharuk} already used condition 
(\ref{AdaptiveControlTheorem1001}) to show asymptotic FWER control of a specific sequence of adaptive 
Bonferroni tests and adaptive SD tests, respectively. Under mild regularity assumptions, Liang and Nettleton 
\cite{liang} showed that the FDR of the adaptive SU test of Storey with altered estimator $\hat n_0 (\lambda) 
= n\frac{1-\hat F_n(\lambda)}{1-\lambda}$ and critical values $\hat \alpha_{i:n} = \frac{i}{\hat n_0(\lambda)} 
\alpha$ is asymptotically controlled at level $\alpha$ for every arbitrary and data dependent selection of the 
tuning parameter $\lambda$ out of a candidate set $\{0=\lambda_0<\ldots<\lambda_m<1\}$. This result may also be 
proved by application of Theorem \ref{AdaptiveControlTheorem1}, but note that Theorem \ref{AdaptiveControlTheorem1} 
works for a much broader class of estimators and also under the reverse martingale model. The conservative consistency 
(\ref{AdaptiveControlTheorem1001}) is a very weak condition. Under mild regularity assumptions the crucial assumption 
(\ref{AdaptiveControlTheorem1001}) is also necessary for (\ref{StatementAdaptiveControlTheorem1}) for Storey type 
estimators (\ref{Storey_estimator_adapt_contr}). 

\begin{proposition}\label{Proposition002}
 Let $(P_n)_n$ be a sequence of reverse martingale models with either 
\begin{itemize}
 \item[(i)] $\frac{N_n}{n}\to 1$ or
 \item[(ii)] $\xi_i \leq \lambda$ for all $i$ and $0 < \eta \leq \frac{N_n}{n}$ for some $\eta$ and all $n$.
\end{itemize}
Consider the $FDR_{P_n}$ of the sequence of SU tests based on (A3) and (\ref{Storey_estimator_adapt_contr}) so that 
(\ref{StatementAdaptiveControlTheorem1}) holds. Let 
$\kappa_n \to 0$ and suppose that $P_n(\{ \hat \alpha_{R:n} = \lambda \}) \to 0$ holds as $n\to \infty$. Then the 
ratio $\frac{\hat n_0}{N_n}\to 1$ converges to one in $P_n$-probability as $n\to \infty$.
\end{proposition}

\begin{remark}[About asymptotic FDR of Storey type SU tests]\label{RemarkAsymptoticVariability}
\hspace{1cm}\\
Consider the reverse martingale model. As long as enough variability of the variables 
$(\epsilon_i 1\{p_i\leq\lambda\})_{i\leq n}$ is present condition (\ref{AdaptiveControlTheorem1001}) 
can be verified. Let $\hat n_0$ be the estimator (\ref{Storey_estimator_adapt_contr}) for some positive 
sequence $\kappa_n$. Let
\begin{equation}
 \hat F_{0,n}(\lambda) = \frac{1}{N_n} \sum_{i=1}^n \epsilon_i 1\{ p_i \leq \lambda \}
\end{equation}
be the cumulative distribution function of true $p$-values. Then a sufficient condition to ensure 
(\ref{AdaptiveControlTheorem1001}) is (\ref{AboutAsymptoticFDR}), where the conditional variances
\begin{equation}\label{AboutAsymptoticFDR}
 Var_{P_n}\left(\hat F_{0,n}(\lambda) \Big| (\epsilon_i)_{i \leq n} \right) 
= E_{P_n}\left( \left(\hat F_{0,n}(\lambda)-\lambda\right)^2 \Big| (\epsilon_i)_{i \leq n} \right)
\to 0
\end{equation}
tends to zero in probability as $n\to\infty$. The corresponding SU tests then have asymptotic FDR control. 
% Example \ref{Example6-1} of Section \ref{SubsectionBlockModel} below provides an example with maximal 
% dependence, where (\ref{AdaptiveControlTheorem1001}) and (\ref{AboutAsymptoticFDR}) are violated. 
\end{remark}

At this point, Theorem \ref{AdaptiveControlTheorem1} can be extended to treat arbitrary $p$-values. Like Benjamini and 
Yekutieli \cite{benjamini_yekutieli} and Blanchard and Roquain \cite{blanchard_roquain_2}, who considered non data 
dependent SU tests for arbitrary dependence structures. Therefore we have to consider more conservative test procedures. 
The adaptive SU test (A4) is based on the critical values (\ref{CritValBR}) of Blanchard and Roquain 
\cite{blanchard_roquain_2} and yields asymptotic FDR control if (\ref{AdaptiveControlTheorem1001}) is satisfied.

\begin{theorem}\label{AdaptiveControlTheorem1Part2}
Let $\widetilde{\mathcal{P}}_n$ be the set of all possible distributions of the $p$-value model (\ref{model002}) for fixed $n$, 
where each $U_1,\ldots, U_n$ is distributed according to the uniform distribution on $(0,1)$. No further distributional 
assumption and no dependence structure is assumed. Again, let $(P_n)_n$ be a sequence of distributions with 
$P_n \in \widetilde{\mathcal{P}}_n$ and $\hat n_{0,n}$ be a sequence of estimators for $N_n$ which fulfill (A2). If 
(\ref{AdaptiveControlTheorem1001}) is  fulfilled, then (\ref{StatementAdaptiveControlTheorem1}) holds for the sequence 
of adaptive SU tests given by (A4).
\end{theorem}

% \begin{remark}
% (a) Condition (\ref{AdaptiveControlTheorem1001}) is the easiest to satisfy for the BI model. The Storey 
% estimator $\hat n_{0,n} = n\frac{1-\hat F_n(\lambda)}{1-\lambda}$ always fulfills (\ref{AdaptiveControlTheorem1001}) 
% for the sequence of BI models for instance. On the other hand, in the maximal dependent case $p_1=\ldots=p_n$, where 
% $p_1 \sim U(0,1)$, (\ref{AdaptiveControlTheorem1001}) does obviously not hold for the Storey estimator and for most 
% other estimators. The maximal dependent case is a possible sequence of distributions in the reverse martingale model 
% with arbitrary dependence. Thus for fixed sequence of estimators $\hat n_{0,n}$ Theorem \ref{AdaptiveControlTheorem1} 
% provides a sufficient condition for possible dependence structures so that asymptotic FDR control holds. \\
% \end{remark}

\subsection{Finite sample results}
Let us now come back to finite sample FDR control. We will give a condition for FDR control for the reverse 
martingale model and the next very useful Lemma offers an exact formula for the FDR of our adaptive tests.  
% % Therefore, the possible adaptive multiple tests are based on the following conditions (A1)-(A3). Let $0 < \lambda < 1$ 
% % and divide $[0,1]$ into two areas, the rejection area $[0,\lambda]$ and the estimation area $[\lambda,1]$.  
% % \begin{itemize}
% %  \item[(A1)] The unknown value $E_P(N)$ is estimated via the empirical cumulative distribution function 
% % 	      $\hat F_n$ on the \textbf{estimation area} $[\lambda,1]$. Let 
% % \end{itemize}
% % 	      \begin{equation}
% % 		\hat n_0 = g((\hat F_n(t))_{\lambda\leq t\leq 1}) > 0 \quad \mbox{almost surely}
% % 	      \end{equation}
% % \begin{itemize}
% % \item[]	      be a positive estimator of $E_P(N)$.
% % \item[(A3)] The multiple test is applied to the \textbf{rejection area} $[0,\lambda]$ with data dependent critical 
% % 	      values
% % \end{itemize}
% % 	      \begin{equation}\label{CritVal002}
% % 	       \hat \alpha_{i:n} = \left(\frac{i}{\hat n_0} \alpha\right)\wedge \lambda, \ 1\leq i\leq n.
% % % 		\quad	\mbox{where } a\wedge b := \min(a,b).
% % 	      \end{equation}
% % Taking the minimum in (\ref{CritVal002}) goes back to Storey et al. \cite{storey_taylor_siegmund} and ensures 
% % that one does not reject $p$-values greater than $\lambda$, 
% % therefore the name rejection area. Statisticians often do not like to reject a hypothesis when the $p$-value 
% % is too high. 
% The next Lemma offers an exact formula for the FDR of our adaptive tests.

\begin{lemma}\label{Theorem003} 
Let $V(\lambda):= \#\{p_i \leq \lambda, \, p_i \mbox{ true}\}$ denote the number of true null hypotheses 
with $p$-values $p_i\leq \lambda$. Under the reverse martingale model, the adaptive SU test with critical 
values (A3) and estimator (A1) fulfills 
% Then for the reverse martingale model and (A1),(A3) the adaptive SU tests given by (\ref{CritVal002}) 
% fulfills 
\begin{equation}\label{OurCondition001}
 E\left( \frac{V}{R} \right) 
= \frac{\alpha}{\lambda} E\left( V(\lambda) \min\left\{\frac{1}{\hat n_0}, \frac{\lambda}{n\hat F_n(\lambda) \alpha}\right\} \right)
\leq \frac{\alpha}{\lambda} E\left( \frac{V(\lambda)}{\hat n_0} \right).
\end{equation}
\end{lemma}

% The very useful equality (\ref{OurCondition001}) is the same equality as in Heesen and Janssen \cite{heesen_janssen} 
% for the BI model. 
This result generalizes Lemma 3.1 of Heesen and Janssen \cite{heesen_janssen}, where the BI model is treated only. 
For the control of the FDR by $\alpha$ one merely has to show 
\begin{equation}\label{verification}
 E\left( \frac{V(\lambda)}{\hat n_0} \right) \leq \lambda.
\end{equation} 
For example the benchmark result of Liang and Nettleton \cite[Theorem 7]{liang} works for estimators $\hat n_0'$ with 
$\hat n_0' \geq \hat n_0(\lambda)$ almost surely with $\hat n_0(\lambda)$  defined in (\ref{Storey_estimator_adapt_contr}) 
with $\kappa_n=\frac{1}{n}$ and for some $\lambda \in [0,1)$. In comparison to that, Lemma \ref{Theorem003} works for the 
class of estimators (A1) and also for the reverse martingale model. Some interesting estimators which do not satisfy 
$\hat n_0' \geq \hat n_0 (\lambda)$ are given Heesen and Janssen \cite{heesen_janssen}.

The following negative result explains first that the use of adaptive SU tests is limited under dependence and further 
results are needed for finite sample FDR control.

\begin{proposition}\label{Proposition001}
 Consider an adaptive SU test based on (A1), (A3) and $\alpha<\lambda$. Assume that the estimator 
% \begin{equation*}
 $p_i \mapsto \hat n_0 $
% \end{equation*}
is non-decreasing for each coordinate $i$. If we have $FDR\leq \alpha$ for all reverse martingale models then 
$\hat n_0 \geq n$ holds and the adaptive critical values $\hat \alpha_{i:n} \leq \frac{i}{n}\alpha$ are dominated 
by the BH critical values. 
\end{proposition}

The adaptive step up test of Storey does not yield FDR control in the reverse martingale and PRDS model. For instance 
Blanchard and Roquain \cite[Theorem 17]{blanchard_roquain} proved that in case of $n=n_0\geq 2$ and $p_1=\ldots=p_n=U$, 
for a uniform distributed $U$ on $(0,1)$, 
$$E\left( \frac{V}{R} \right) = \min\left( \alpha n(1-\lambda), \lambda \right)
$$
holds for the Storey estimator (\ref{Storey_estimator_adapt_contr}) with $\kappa_n=\frac{1}{n}$. This result corresponds 
to Lemma \ref{Theorem003}. Note that condition (\ref{AdaptiveControlTheorem1001}) is violated since 
$\frac{1}{\hat n_0} =1-\lambda$ holds on $\{U\leq\lambda\}$ for the Storey estimator (\ref{Storey_estimator_adapt_contr}). 

% \begin{example}\label{Example_PRDS_Storey}
%  Let $n=n_0\geq 2$ and $p_1=\ldots=p_n=U$, where $U$ is distributed according to the uniform distribution on $(0,1)$ 
% and (\ref{AdaptiveControlTheorem1001}) is violated for the Storey estimator (\ref{Storey_estimator_adapt_contr}) with 
% $\kappa_n=\frac{1}{n}$. This is a possible case of the reverse martingale and PRDS model. Observe that
% \begin{equation*}
%  \frac{1}{\hat n_0} = 1-\lambda \quad \mbox{and}\quad \frac{\lambda}{n\hat F_n(\lambda)\alpha}=\frac{\lambda}{n\alpha}
% \end{equation*}
% holds on $\{U\leq \lambda\}$ for the Storey estimator (\ref{Storey_estimator_adapt_contr}).
% Thus by Lemma \ref{Theorem003} 
% \begin{equation*}
%  E\left( \frac{V}{R} \right) 
% = \frac{\alpha}{\lambda} E\left( n 1\{ U\leq \lambda \} \min\left\{1-\lambda,\frac{\lambda}{n\alpha}\right\} \right)
% = \min\left( \alpha n(1-\lambda), \lambda \right)
% \end{equation*}
% and hence $FDR > \alpha$ holds for $\lambda > \alpha$ and $n(1-\lambda)>1$.
% \end{example}

\subsection{Case of a block model}\label{SubsectionBlockModel}
Finally the following modified adaptive SU test is considered when mild additional dependence assumptions are present. 
Suppose that the $p$-values can be divided by
\begin{equation}\label{Grouping001}
 \{ p_1,\ldots, p_n\} = \bigcup_{i=1}^k G_i
\end{equation}
in $k$ disjoint blocks or groups $G_i$. 
Suppose the reverse martingale condition for $(p_1,\ldots, p_n)$. Assume in addition that for each group the subset 
$\widetilde G_i \subset G_i$ corresponds to uniformly distributed $p$-values given by true null hypotheses. Below let 
the groups $\widetilde G_i$, $1\leq i\leq k$, be conditionally independent given the signs $\epsilon$ whereas within 
group $\widetilde G_i$ a reverse martingale dependence structure is allowed. 

\begin{remark}
 In practice the model may have the following meaning for genome data. $\widetilde G_i$ may stand for independent 
portions of true $p$-values which may come from different chromosomes. The $p$-values of $\widetilde G_i$ may 
be reverse martingale dependent, for instance some of them may be equal. 
\end{remark}

Consider the maximal size $m$ of the groups
\begin{equation}\label{Grouping002}
 m:= \max_{i\leq k} |G_i| \quad \mbox{and } n=mk-r
\end{equation}
with a remainder $r\geq 0$. Furthermore, let us assume that the number of true $p$-values $N$ is almost surely lower 
bounded by $N_{min}$. For the tuning parameters $0< \lambda < 1$ and $\kappa \geq 1$ the modified Storey estimator 
\begin{equation}\label{Grouping004}
 \hat n_0(\kappa) := n \frac{1- \hat F_n(\lambda) +\frac{\kappa}{n}}{1-\lambda}
\end{equation}
with $\kappa_n = \frac{\kappa}{n}$ is introduced. Again (\ref{Grouping004}) can be improved by the factor $(1-\lambda^k)$, 
i.e. also $(1-\lambda^k)\hat n_0(\kappa)$ will work. We show that the step up test with estimated critical values
\begin{equation}\label{Grouping003}
 \hat \alpha_{i:n} = \left( \frac{i}{\hat n_0(\kappa)}\alpha\right) \wedge \lambda
\end{equation}
yields FDR control , $FDR\leq \alpha$, under the present block wise dependence model for all $\kappa \geq m+r+(n-N_{min})$. 
If the groups are balanced, i.e. $|G_1|=\ldots=|G_k|$ holds, 
then $r$ vanishes and the best fit is expected. Of course $\hat n_0(\kappa) > n$ may happen for large $r$ in the 
unbalanced case and there would then be no advantage in comparison with the BH test when the $p$-values are all 
independent.

\begin{theorem}\label{Theorem004}
 Consider the reverse martingale model for the $p$-values. Assume that the $p$-values can be divided in $k\geq 2$ 
disjoint groups, see (\ref{Grouping001}) and (\ref{Grouping002}) above. Moreover, assume that $N\geq N_{min}$ holds 
almost surely for a lower bound $N_{min}>0$. Let conditionally on the signs $\epsilon$ the groups $\widetilde G_1, 
\ldots, \widetilde G_k$ of the true $p$-values be independent. If $\kappa \geq m+r+(n-N_{min})$ holds, then the 
modified adaptive SU test with critical values (\ref{Grouping003}) and estimator (\ref{Grouping004}) has finite 
sample FDR control, i.e. $FDR \leq \alpha$, and (\ref{verification}) holds. 
\end{theorem}

Gontscharuk \cite{gontscharuk} considered a similar block model which leads to dependent $p$-values and an adaptive 
Bonferroni type procedure with asymptotic FWER control. Theorem \ref{Theorem004} works for finite $n$. 
% , where the $p$-values 
% also consist of independent blocks but with arbitrary dependence structure within each block. Gontscharuk 
% \cite{gontscharuk} then introduces further conditions so that the $p$-values are weakly dependent and showed that basically under 
% condition (\ref{AdaptiveControlTheorem1001}) and weak dependency an adaptive Bonferroni type procedure yields asymptotic 
% FWER control. In short, the $p$-values are weakly dependent if the Glivenko-Cantelli Theorem holds for the empirical 
% cumulative distribution function of true $p$-values. Theorem \ref{Theorem004} works for finite $n$.

If the group structure is known and balanced with $|G_1|=\ldots=|G_k|=m$, then Guo and Sarkar \cite{guo_sarkar} propose 
an adaptive multiple test with FDR control under PRDS within each group. The ingredients are based on the Storey type 
estimator (\ref{Storey_estimator_adapt_contr}), where $\lambda$ depends on the number of blocks and $\kappa=\frac{m}{n}$. 
However, every rejected $p$-values has to be less than or equal to $\frac{k}{\hat n_0(\lambda)}\alpha$ in comparison to 
$\frac{n}{\hat n_0(\lambda)}\alpha$ for the adaptive SU test considered in Theorem \ref{Theorem004}.
% % Ausführlicher: 
% Let $\widetilde p_i$ be the smallest $p$-value 
% of the $i$-th group and $\hat n_0(\lambda)$ the Storey estimator defined in (\ref{Storey_estimator_adapt_contr}) with 
% $\kappa=\frac{m}{n}$ and $(2k+3)^{-\frac{2}{k+2}}\leq\lambda$. Then their test rejects every $p$-value $p_i$ which is 
% less than or equal to $\frac{R_G}{\hat n_0(\lambda)}\alpha$, where $R_G=\max\{i : \widetilde p_i \leq \frac{i}{\hat n_0(\lambda)}\alpha\}$ 
% is determined in the SU sense. Since $R_G\leq k$ and $n=km$, the effective threshold $\frac{R_G}{\hat n_0(\lambda)}\alpha$ 
% of the test may be very small for large $m$ in comparison to the effective threshold $\frac{R}{\hat n_0(\lambda)}\alpha$ 
% of the adaptive SU test considered in Theorem \ref{Theorem004}, where $0\leq R \leq n$. 

As mentioned above the estimator (\ref{Grouping004}) may produce very conservative SU tests for independent $p$-values. 
However, Theorem \ref{Theorem004} is designed for block dependent $p$-values and we will see by the inspection of the 
FWER that this procedure can not always be improved. The necessary calculations for Example \ref{Example6-1} are included 
in the proof of Theorem \ref{Theorem004}. 
% Moreover, simulations indicate that $\kappa \approx m$ is necessary for FDR control in the submodel of Theorem \ref{Theorem004}. 

\begin{example}\label{Example6-1}
 Consider $k$ blocks of size $m$ with $N=n=mk$ true $p$-values. Suppose that for each block all $p$-values coincide with the 
same uniformly distributed random variable whereas the blocks are independent. The model has both the PRDS and the reverse 
martingale property. When $\alpha \leq \frac{\lambda}{k(1-\lambda)}$ holds then the choice $\kappa=m$ for the procedures 
(\ref{Grouping004}) and (\ref{Grouping003}) yields 
\begin{equation}
 FDR = FWER = \alpha(1-\lambda^k).
\end{equation}
Here the modified estimator $(1-\lambda^k)\hat n_0(\kappa)$ attains the FWER bound $\alpha$ and we obtain a sharp result for 
the block model. 
\end{example}

We conducted a small Monte-Carlo simulation with 100.000 repetitions to explore the situation beyond the case 
of Example \ref{Example6-1} for the adaptive SU test with critical values (\ref{Grouping003}) and estimator 
(\ref{Grouping004}) with $\alpha=0.05$ at size $n=100$. All false $p$-values are set to $0$. Consider a setting of $k=5$ groups 
with $m=20$ $p$-values and 16 equal true $p$-values in each group. The choices of $\kappa=1,16,20$ lead to 
$FDR \approx 0.0886, 0.0476, 0.0438$. Furthermore, the group frequencies 25,25,20,15,15 with 20,20,16,12,12 
equal true $p$-values within theses groups and the choices of $\kappa=1,12,20,25$ lead to $FDR \approx 0.0921, 
0.0567, 0.0446, 0.0385$. Since the number of true $p$-values within each group is unknown, the simulation and 
Example \ref{Example6-1} indicate that here $\kappa \approx m$ is an appropriate tuning parameter for the 
adaptive SU test. Moreover, a simulation with equal group frequencies under the global intersection hypothesis 
shows that $\kappa$ slightly smaller than $m$ already yields $FDR > \alpha$. For $k=10$ groups of equal size 
with $m=100$ and only true $p$-values in each group, the choice of $\kappa=97$ yields $FDR\approx0.051$.

\section{Technical results and proofs}
\label{Theoretical_results}

\begin{lemma}\label{LemmaCentral}
(a) Let $0<\hat\alpha_{1:n}\leq \ldots \leq\hat\alpha_{n:n}\leq \lambda < 1$ be data dependent critical values 
\begin{equation} 
 \hat\alpha_{i:n} = g_i((\hat F_n(t))_{t\geq\lambda}), \quad i=1,\ldots, n,
\end{equation}
given by measurable functions $g_i$ and introduce $\hat\alpha_{0:n}=\hat \alpha_{1:n}$. Moreover define 
$\gamma(i) := n \hat\alpha_{i:n}$. Then 
\begin{equation}\label{LemmaCentral001}
 E\left(\frac{V}{\gamma(R)}\right) = \frac{E(N)}{n}
\end{equation}
holds for the corresponding adaptive SU tests under the reverse martingale model (including the BI model).\\
% (b)  Let $0<\alpha_{0:n}=\alpha_{1:n}\leq \ldots \alpha_{n:n}\leq \lambda < 1$ be deterministic critical values and let 
% $\gamma(i) := n \alpha_{i:n}$. Then we have ''$\leq$`` in (\ref{LemmaCentral001}) for the PRDS model and 
% ''$\geq$`` in (\ref{LemmaCentral001}) for the NRDS model.\\
% % (b) Under the assumptions of (a) we have ''$\leq$`` in (\ref{LemmaCentral001}) for the PRDS model and 
% % ''$\geq$`` in (\ref{LemmaCentral001}) for the NRDS model.\\
(b) Let $\hat\rho(i)=\hat\rho(i,(\hat F_n(t))_{0\leq t\leq 1}) > 0$, $i=0,\ldots, n$, be non-decreasing in $i$ and let 
$\gamma(i):= \alpha \hat\rho(i)$. Moreover, assume that $\nu$ is a probability measure on $(0,\infty)$ and define the 
data dependent critical values via 
\begin{equation}\label{LemmaCentralCritVal001}
 \hat\alpha_{i:n} = \frac{\alpha}{n} \int_{0}^{\hat\rho(i)} x d\nu(x).
\end{equation}
Then ''$\leq$`` holds in (\ref{LemmaCentral001}) for the corresponding adaptive SU test with critical values 
(\ref{LemmaCentralCritVal001}) for arbitrary dependent variables $(\epsilon_i, U_i, \xi_i)_{i\leq n}$.
\end{lemma}

\begin{remark}\label{RemarkCentral}
(a) Lemma \ref{LemmaCentral} (a) also applies to deterministic critical values $0<\alpha_{1:n}\leq \ldots \alpha_{n:n}< 1$ 
if we put $\lambda = \alpha_{n:n}$. In this case, the reverse martingale assumption (\ref{MartingaleDef}) can be weakened 
in order to prove (\ref{LemmaCentral001}). It is only necessary to assume that  (\ref{MartingaleDef}) is a reverse martingale 
w.r.t. the discrete parameter set $I:=\{\alpha_{1:n},\ldots,\alpha_{n:n},1\}$ and $t\in I$. \\
(b) Storey et al. \cite{storey_taylor_siegmund} already used martingale arguments which have been outlined by 
Scheer \cite{scheer}. \\
(c) In case of deterministic critical values $0<\alpha_{1:n}<\ldots<\alpha_{n:n}<1$ we obtain the inequality ''$\leq$`` 
in (\ref{LemmaCentral001}) under the PRDS model. The proof follows straightforward classical lines, see Heesen \cite{heesen} 
and Meskaldji et al. \cite{meskaldji}. 
\end{remark}

\par\medskip\noindent
{\sc Proof of Lemma \ref{LemmaCentral}.}
(a) Observe that the SU test can be represented by the reverse stopping time 
\begin{equation*}
 \tau = \sup \{ \hat \alpha_{i:n}, i=1,\ldots,n \, : \, p_{i:n} \leq \hat \alpha_{i:n} \}\vee \hat \alpha_{1:n},
\end{equation*} 
which is adapted to the reverse Filtration $(\mathcal{F}_t)_{0 < t \leq 1}$ and where $\sup \emptyset := 0$.
Then every $p$-value $p_i \leq \tau$ is rejected. For $V(t):= \#\{p_i \leq t, \, p_i \mbox{ true}\}$ and 
$R(t)=n\hat F_n(t)$ for $0\leq t \leq 1$ we have $\frac{V}{\gamma(R)}  = \frac{V(\tau)}{\gamma(R(\tau))} $. 
Conditioned under $\mathcal{F}_\lambda$ the critical values $\hat\alpha_{1:n}, \ldots, \hat \alpha_{n:n}$ are fixed 
and therefore, $\tau$ is a discrete stopping time w.r.t. the reverse martingale (\ref{MartingaleDef}) for the periode 
$\hat \alpha_{1:n}\leq t \leq \lambda$. Furthermore, observe that $\hat \alpha_{R(\tau):n} = \tau$ holds 
if $R>0$ since $R(\tau)=R$. Thus by (\ref{MartingaleDef}) and the discrete version of the optional stopping theorem 
\begin{eqnarray*}
 E\left( \frac{V}{\gamma(R)} \Big| \mathcal{F}_\lambda \right) 
&=& E\left( \frac{V(\tau)}{\gamma(R(\tau))} \Big| \mathcal{F}_\lambda \right) 
= E\left( \frac{V(\tau)}{n\hat \alpha_{R(\tau):n}} 1\{ R(\tau)>0 \} \Big| \mathcal{F}_\lambda \right) \\
&=& \frac{1}{n} E\left( \frac{V(\tau)}{\tau} 1\{ R(\tau)>0 \} \Big| \mathcal{F}_\lambda \right) 
= \frac{1}{n} E\left( \frac{V(\tau)}{\tau} \Big| \mathcal{F}_\lambda \right)\\
&=&  \frac{1}{n} E\left( \frac{V(\lambda)}{\lambda} \Big| \mathcal{F}_\lambda \right)
\end{eqnarray*}
holds and integration yields
\begin{equation*}
 E\left( \frac{V}{\gamma(R)} \right) = \frac{1}{n} E\left( \frac{V(\lambda)}{\lambda} \right) 
=  \frac{1}{n} E\left( \frac{V(1)}{1} \right) = \frac{E(N)}{n}.
\end{equation*}
(b) Applying the technique of the proof of Lemma 3.2 of Blanchard and Roquain \cite{blanchard_roquain} yields 
\begin{equation*}
 E\left( \frac{V}{\gamma(R)} \Big| \epsilon \right) 
= \sum_{i\, : \, \epsilon_i=1} E\left( \frac{1\{ p_i \leq \hat \alpha_{R:n} \}}{\alpha \hat\rho(R)} \Big| \epsilon \right)
\leq \frac{N}{n}.
\end{equation*}
There, the technique is formulated for deterministic critical values, but observe that it also works for data dependent 
critical values. 
$\hfill \square$

\par\medskip
When the proof was finished we came across the early paper of Meskaldji et al. \cite{meskaldji} which covers the 
special non-adaptive case of our technical Lemma \ref{LemmaCentral} (b). For deterministic critical values their 
proof also follows the lines of Blanchard and Roquain \cite{blanchard_roquain_2}.

\subsection{ Proofs of Section \ref{SecExamplesRevMart}}

{\sc Proof of Example \ref{NeuesExamplePosDep}.} 
Below the proof of part (a) is sketched. The calculations for (b) are similar. Recall from Finner et al. \cite{finner_dickhaus_roters}, 
p. 604, the FDR formula for $n=n_0=2$ 
\begin{equation}
\begin{split}
  &FWER = FDR(2) \\
&\ \ = \alpha \left[ P\left(R\geq1|p_1\leq\frac{\alpha}{2}\right) - P\left(R\geq2|p_1\leq\frac{\alpha}{2}\right) + P\left(R\geq2|p_1\leq\alpha\right)\right]
\end{split}
\end{equation}
since $(X_1,X_2)\overset{d}{=} (X_2,X_1)$ holds for the normal sample . For $\alpha=\frac{1}{2}$ we have by conditioning w.r.t. $X_1=x_1$ 
\begin{eqnarray*}
 P\left(R\geq2, p_1\leq \frac{1}{2}\right) &=& P\left(p_1\leq \frac{1}{2},p_2\leq \frac{1}{2}\right) = P(X_1\leq0,X_2\leq0)  \\
&=& \int_0^\infty \Phi(x_1)\phi(x_1)dx_1,
\end{eqnarray*}
where $\phi=\Phi'$. Similarly, 
\begin{equation}
 P\left(R\geq2,p_1\leq \frac{1}{4}\right) = \int_{-\Phi^{-1}(\frac{1}{4})}^\infty \Phi(x_1)\phi(x_1)dx_1.
\end{equation}
Notice that 
\begin{equation}
 \int_a^\infty\Phi(x)\phi(x)dx = \frac{1}{2}\Phi(x)^2 \big|_a^\infty = \frac{1}{2} [1-\Phi(a)^2]
\end{equation}
follows. Thus, $P(R\geq2|p_1\leq\frac{1}{2}) = \frac{3}{4}$ and $P(R\geq2|p_1\leq\frac{1}{4}) = \frac{7}{8}$ holds which implies the result. 
\hfill $\square$

\par\bigskip\noindent
{\sc Proof of Example \ref{ExampleRevMartingale}.} 
Part (b) and (c) are obvious. To prove (a) define 
\begin{equation*}
 M_t^{(i)} = \frac{1\{ Z_i \leq t \}}{H(t)}, \quad \mbox{for } t \in \{ H>0\}.
\end{equation*}
In case $n=1$ it is well known that $M_t^{(1)}$ is a reverse martingale w.r.t. $G_t= \sigma(M_s^{(1)}\, : \, s\geq t)$. 
In case $n>1$ let us prove that $M_t^{(1)}$ is a reverse martingale w.r.t. 
$\mathcal{F}_t := \sigma((M_s^{(j)})_{s\geq t}, 1\leq j \leq n)$. Obviously, $E(M_t^{(1)}|\mathcal{F}_s)=0=M_s^{(1)}$ holds 
if $Z_1 > s$. Otherwise, $Z_1\leq s$ implies $X_1\leq s$ and $Y \leq s$. Thus $1\{ Z_i \leq \tau \} = 1\{ X_i \leq \tau \}$ 
holds for $i \geq 2$ and all $\tau \geq s$. Let $f_i : [s,\infty) \to \{0,1\}^{[s,\infty)}$ be a possible path of 
$\tau \mapsto 1\{X_i \leq \tau \}$, $\tau \geq s$. If $1\{Z_1 \leq s \}=1$ holds we have 
\begin{eqnarray*}
&& \hspace{-0.5cm} E( M_t^{(1)} | \mathcal{F}_s ) \\
&&=  E( M_t^{(1)} |  1\{Z_1 \leq s \}=1, (1\{ X_i\leq \tau \})_{\tau \geq s} = (f_i(\tau))_{\tau \geq s}, i=2,\ldots, n)\\
&&= E( M_t^{(1)} |  1\{Z_1 \leq s \}=1) \\
&&= E\left( M_t^{(1)} \Big|  M_s^{(1)} = \frac{1}{H(s)}\right) = M_s^{(1)}.
\end{eqnarray*}
Above we used that $1\{Z_1 \leq s \}$ and $(X_2,\ldots, X_n)$ are independent. Observe that the time change 
\begin{equation*}
 u \mapsto M_{H^{-1}(u)}^{(i)} = \frac{1\{ H(Z_i) \leq u \}}{u} 
\end{equation*}
by the inverse distribution function $H^{-1}$ preserves the reverse martingale.
\hfill $\square$

\subsection{Proofs of Section \ref{SecIneq}}

 \begin{lemma}\label{ApplicationsLemma001}
(a) Assume the PRDS model or reverse martingale model (which include the BI model). Consider the SU test 
with deterministic critical values (\ref{stepup001}) and suppose that there exist some $k$ and $0<c<1$ with 
$\alpha_{j:n} \leq \frac{jc}{n}$ for all $j\leq k$.
\begin{itemize}
 \item[(i)] Then $FDR \leq \frac{cE(N)}{n} + P(R>k)$ holds.
 \item[(ii)] Suppose that the assumption holds for $k=n$. If in addition $\alpha_{j:n} < \frac{jc}{n}$ holds for a 
fixed $j$ with $P(R=j, V>0) > 0$, then $FDR < \frac{cE(N)}{n}$ follows for the BI model.
\end{itemize}
(b) Assume the reverse martingale model, consider the SU test with critical values (\ref{stepup001}) 
and suppose that $\alpha_{j:n} \geq \frac{jc}{n}$ holds for all $j \leq n$. 
\begin{itemize}
 \item[(i)] Then $FDR \geq \frac{cE(N)}{n}$ holds.
 \item[(ii)] If in addition $\alpha_{j:n} > \frac{jc}{n}$ holds for some $j$ with 
$P( R=j, V>0 ) > 0$, then $FDR > \frac{cE(N)}{n}$ holds at least under the BI model.
\end{itemize}
\end{lemma}

% \par\bigskip
% {\sc Proof of Lemma \ref{ApplicationsLemma001}.}
\begin{proof}
(a)(i) Let $\gamma(i)= n \alpha_{i:n}$ and observe that $\frac{\gamma(j)}{j} \leq c$ holds for all $j \leq k$. 
Then Lemma \ref{LemmaCentral} (a) and Remark \ref{RemarkCentral} (c) imply 
\begin{eqnarray*}
 FDR 
% &\leq& E\left(\frac{V}{R} 1\{R \leq k\}\right) + P(R > k)\\
 &\leq& E\left(\frac{V}{\gamma(R)}\frac{\gamma(R)}{R} 1\{R \leq k\}\right) + P(R > k)\\
 &\leq& \frac{cE(N)}{n} + P(R > k)
\end{eqnarray*}
for the PRDS and reverse martingale model, respectively.\\
(ii) By (i) we already know $FDR \leq \frac{cE(N)}{n}$ and in case of equality we have $0 = E(\frac{cV}{\gamma(R)}-\frac{V}{R})$. 
But observe that $\frac{cV}{\gamma(R)} - \frac{V}{R} \geq 0$ holds. Thus, 
\begin{equation}\label{PlusPlus}
 E\left(V\left(\frac{c}{\gamma(R)} - \frac{1}{R}\right)\right)  \geq \left( \frac{c}{\gamma(j)}- \frac{1}{j} \right) \cdot P(R=j, V>0)
\end{equation}
follows by our assumptions, a contradiction. \\
% (ii) Observe that $\frac{cV}{\gamma(R)} - \frac{V}{R} \geq 0$ holds. By (i) we already know $FDR \leq \frac{cE(N)}{n}$. 
% In case of equality we have $0 = E(\frac{cV}{\gamma(R)} - \frac{V}{R})$ which implies 
% \begin{equation}\label{PlusPlus}
%  E\left(V\left(\frac{c}{\gamma(R)} - \frac{1}{R}\right)\Big| n_0, \bar f \right) = 0 \quad a.e.,
% \end{equation}
% for the conditional expectation given $N=n_0$ and the portion of false $p$-values $\bar f := (p_j : \epsilon_j=0)$.
% By our assumptions there exists $n_0 \geq n+1-j$ with $P(N=n_0) > 0$ whereas $\bar f$ is an arbitrary but fixed vector 
% of size $n-n_0 \leq j-1$. 
% % If $p_1,\ldots, p_{n_0}$  are true $p$-values 
% Since $(p_j : \epsilon_j=1)$ are the true $p$-values, we see that the set
% $$B(n_0, \bar f) := \{R=R((p_j : \epsilon_j=1), \bar f) = j\}$$
% has positive probability in the BI model and $V$ is positive on that set. By our definitions $(\frac{c}{\gamma(R)} - \frac{1}{R})1\{V > 0\}$ is positive on $B(n_0, \bar f)$ which contradicts 
% (\ref{PlusPlus}).\\
(b)(i) Observe that $\frac{\gamma(j)}{j} \geq c$ holds for all $j\leq n$. Thus Lemma \ref{LemmaCentral} implies
\begin{equation*}
 FDR = E\left( \frac{V}{\gamma(R)}\frac{\gamma(R)}{R} \right) 
\geq c E\left( \frac{V}{\gamma(R)} \right) \geq \frac{cE(N)}{n}
\end{equation*}
for the reverse martingale model.\\
(ii) Observe that $\frac{cV}{\gamma(R)} - \frac{V}{R} \leq 0$ and $\frac{c}{\gamma(j)} - \frac{1}{j} < 0$ hold and 
the assertion follows in the same way as in (a) (ii). 
% (ii) Now observe that $\frac{cV}{\gamma(R)} - \frac{V}{R} \leq 0$ holds and in case of equality we have 
% $0 = E(\frac{cV}{\gamma(R)} - \frac{V}{R})$ which implies 
% \begin{equation}\label{PlusPlus2}
%  E\left(V\left(\frac{c}{\gamma(R)} - \frac{1}{R}\right)\Big| n_0, \bar f \right) = 0 \quad a.e.
% \end{equation}
% Again by our assumptions there exists $n_0 \geq n+1-j$ with $P(N=n_0) > 0$ whereas $\bar f$ is an arbitrary but fixed 
% vector of size $n-n_0 \leq j-1$.
% Again we see that $B(n_0, \bar f) := \{R=R((p_j : \epsilon_j=1), \bar f) = j\}$ has positive probability in the BI model and $V$ is positive on that set. 
% By our definitions $(\frac{c}{\gamma(R)} - \frac{1}{R})1\{V > 0\}$ is negative on $B(n_0, \bar f)$ which contradicts 
% (\ref{PlusPlus2}).
% %   \hfill $\square$
\end{proof}

\par\bigskip\noindent
{\sc Proof of Proposition \ref{AppicationsLemmaDirect}.}
The proposition is a direct application of Lemma \ref{LemmaCentral} (a) and Remark \ref{RemarkCentral} (c) 
with deterministic critical values. Again let $\gamma(i)=n\alpha_{i:n}$, $0\leq i \leq n$. For the reverse 
martingale and PRDS models we have 
\begin{equation}\label{AppicationsLemmaDirectProof001}
 FDR = E\left( \frac{V}{\gamma(R)} \frac{\gamma(R)}{R} \right) 
\leq E\left( \frac{V}{\gamma(R)} \right)  \max_{i \leq n} \frac{n\alpha_{i:n}}{i} 
\leq \frac{E(N)}{n} \max_{i \leq n} \frac{n\alpha_{i:n}}{i} 
\end{equation}
by Lemma \ref{LemmaCentral} (a) and Remark \ref{RemarkCentral} (c). Under the assumptions of (b)(i) the first 
inequality in (\ref{AppicationsLemmaDirectProof001}) is actually a strict inequality ''$<$``. The other 
inequalities follow analogously. \hfill $\square$

\par\bigskip\noindent
{\sc Proof of Lemma \ref{ApplicationsLemma002}.}
(a) In this proof we may always choose the Dirac uniform case $\bar f = (p_j : \epsilon_j=0) = (0,\ldots, 0)$ for the 
false $p$-values. Let $N = n_0 = n+1-j$ be deterministic. Thus by Lemma \ref{LemmaCentral} (a)
\begin{eqnarray*}
 \frac{n_0}{n} &=& E\left( \frac{V}{\gamma(R)} \Big| n_0, \bar f \right) 
=  E\left( \frac{V}{R}\frac{R}{\gamma(R)} \Big| n_0, \bar f  \right)\\
&\leq& E\left( \frac{V}{R} \Big| n_0, \bar f  \right) \frac{j}{\gamma(j)}
\end{eqnarray*}
since $R \geq j$ holds when $V$ is positive and $j \mapsto \frac{j}{\gamma(j)} = \frac{j\alpha}{n\alpha_{j:n}}$ 
is non-increasing. By our assumption the inequality $E(\frac{V}{R} | n_0, \bar f ) \leq \alpha$ proves the result.\\ 
(b) In that case we have $\frac{j}{\gamma(j)} \geq \frac{k}{\gamma(k)} > \frac{k+1}{\gamma(k+1)}$ for $j \leq k$. 
Obviously, $P(R=k+1 | n_0, \bar f ) >0$ holds for $N = n_0 = n+1-j$ with $j \leq k$. Thus we have strict inequality in the proof of 
part (a) and $\frac{n_0}{n} < \frac{j\alpha}{\gamma(j)}$ follows. \hfill $\square$

\par\bigskip\noindent
{\sc Proof of Corollary \ref{ApplicationsCorollary001}.}
(a) is a special case of Lemma \ref{ApplicationsLemma002}.\\
(b) Assume the multiple test is not a BH test. Take the first value $k$ with
\begin{equation*}
 \frac{\alpha}{n} = \frac{\alpha_{1:n}}{1} = \ldots = \frac{\alpha_{k:n}}{k} < \frac{\alpha_{k+1:n}}{k+1}.
\end{equation*}
Then Lemma \ref{ApplicationsLemma002} (b) implies $\alpha_{1:n} < \frac{\alpha}{n}$ which contradicts our assumption.\\ 
(c) On the other hand $\alpha_{1:n}=\frac{\beta}{n}$ implies $\frac{\alpha_{j:n}}{j} \geq \frac{\beta}{n}$ and 
$FDR \geq \frac{\beta E(N)}{n}$ by Lemma \ref{ApplicationsLemma001} (b). \hfill $\square$

% 
% % % Neuer Aufbau
% 

\medskip 
The following technical tools of Lemma \ref{LemmaProp3-2ab} and Remark \ref{RemarkRecursion} supplement the inequalities 
of Section \ref{SecIneq}. 

% Prop 3.2 a,b
\begin{lemma}\label{LemmaProp3-2ab}
 Consider SU tests based on the critical values (\ref{CritValBeliebig}) with $0<a<b$ with fixed value $b$ 
and an adjustment of $a$. \\
(a) Under the reverse martingale model and the Dirac uniform configuration DU$(n_0)$ the FDR is given by 
\begin{equation}\label{FDR_BGS}
 g_{a,b}(n_0) := FDR_{DU}(n_0) = \frac{\alpha n_0}{n+b} + \frac{a E_{DU}(V|n_0)}{n+b},
\end{equation}
with ''$\leq$`` under PRDS. Then 
\begin{equation}\label{proposition_29Jan001Formel001}
 \sup_{P \in \mathcal{P}_{BI}} FDR_{(b,a)} = \max_{1\leq n_0 \leq n} g_{a,b}(n_0)
\end{equation}
holds. \\
(b) The following inequality holds for the expected number of false rejections $E_{DU}(V|n_0)$ under DU$(n_0)$ 
\begin{equation*}
 n_0 \alpha_{n+1-n_0:n} \leq E_{DU}(V|n_0).
\end{equation*} 
Note that this statement holds without any dependence assumption on $U_1,\ldots, U_n$.
\end{lemma}

\begin{proof}
 Part (a) follows from (\ref{LemmaCentral001}) which reads under DU with $n_1 = n - n_0$ as
\begin{equation*}
 \frac{n_0}{n} = \frac{1}{n\alpha} E\left( \frac{V(n+b)-(n_1+V)aV}{n_1+V} \right).
\end{equation*}
% (b) Without restrictions we assume that $p_1$ is true. Put $p=(p_1,\ldots,p_n)$ and $p^{(0)}=(0,p_2,\ldots,p_n)$ and 
% $V_{n_0}=V$. Then we  have $E(V_{n_0}) = 
Equation (\ref{proposition_29Jan001Formel001}) holds since the DU configuration is least favorable. \\
(b) Without restrictions we assume that the true $p$-values are given by $p_1,\ldots,p_{n_0}$. Then we obtain 
\begin{eqnarray*}
 E_{DU}(V|n_0) &=& \sum_{i=1}^{n_0} E_{DU}(1\{p_i \leq \alpha_{R:n}\}|n_0) \\
&\geq& \sum_{i=1}^{n_0} E_{DU}(1\{p_i \leq \alpha_{n+1-n_0:n}\}|n_0) = n_0 \alpha_{n+1-n_0:n}
\end{eqnarray*}
since $R \geq n+1-n_0$ holds on $\{p_i \leq \alpha_{R:n}\}$ for the DU configuration for all true $p_i$.
\end{proof}

By different methods Scheer \cite{scheer} obtained (\ref{FDR_BGS}) for $a=1-\alpha$.

\begin{remark}\label{RemarkRecursion}
 The expected number of false rejections $h(n_0,\alpha) = E_{DU}(V^{BH,\alpha}|n_0)$ is easy to compute by the following 
recursion 
% , see also Finner and Roteres \cite{finner_roters} p.991. It is given by 
$h(1,\alpha) = \alpha$ and  
\begin{equation}\label{RecursionInduction}
 h(n_0, \alpha) = \frac{n_0 \alpha}{n}[h(n_0-1,\alpha) + n-n_0+1] 
\end{equation}
 for the BH SU test which equals by induction
\begin{equation}\label{RecursionInduction2}
 h(n_0,\alpha) = \frac{n_0 !}{n^{n_0-1}}\alpha^{n_0} + \sum_{j=1}^{n_0-1} \frac{n_0!}{j!} \left(\frac{\alpha}{n}\right)^{n_0-j}(n-j).
\end{equation}
Formula (\ref{RecursionInduction2}) is due to Finner and Roters \cite[p. 991]{finner_roters} which is now a direct consequence 
of our equalities, see (\ref{RecursionInductionProof}).
\end{remark}

\par\medskip\noindent
{\sc Proof of the Recursion (\ref{RecursionInduction}) of Remark \ref{RemarkRecursion}.}
Consider a true $p$-value, say $p_1$, of $p=(p_1,\ldots,p_n)$. Put $p^{(1)} =(0,p_2,\ldots,p_n)$ and let $R(p), R(p^{(1)})$ be 
the number of rejections. Simple calculations show that $R(p)= R(p^{(1)})$ holds on the set $\{ p_1 \leq \alpha_{R(p):n} \}$.
Observe, when $p_1$ is rejected, then $p_1$ could also be zero. Moreover, we have 
$\{ p_1 \leq \alpha_{R(p):n} \} = \{ p_1 \leq \alpha_{R(p^{(1)}):n} \}$ in any case for SU tests which implies
\begin{equation}
 E(1\{p_1 \leq \alpha_{R(p):n}\}|n_0) =  E( \alpha_{R(p^{(1)}):n}|n_0) = \frac{\alpha}{n} E( R(p^{(1)})|n_0) 
\end{equation}
by Fubini's theorem and $E(V|n_0)=\alpha\frac{n_0}{n}E(R(p^{(1)})|n_0)$. Under DU$(n_0)$ we obtain 
% $R(p^{(1)})= V(p^{(1)}+n-n_0$
\begin{equation}\label{RecursionInductionProof}
 E_{DU}( R(p^{(1)}) |n_0) =  E_{DU}( V(p^{(1)})+n-n_0 |n_0) =  E_{DU}( V(p)+n-n_0+1 |n_0-1) 
\end{equation}
and the recursion. Formula (\ref{RecursionInduction2}) follows by induction.
\hfill $\square$

\begin{remark}\label{RemarkzumCentralLemmaProof}
 Under the BI model the proof of statement (\ref{LemmaCentral001}) can be simplified as follows. Consider deterministic 
critical values. Using the proof of Remark \ref{RemarkRecursion} above we may again conclude by Fubini's Theorem 
\begin{eqnarray*}
 E\left( \frac{1\{p_1 \leq \alpha_{R:n}\}}{\gamma(R)} \Big| \epsilon_1=1 \right) 
 &=& E\left( \frac{1\{p_1 \leq \alpha_{R(p^{(1)}):n}\}}{\gamma(R(p^{(1)}))} \Big| \epsilon_1=1 \right) \\
&=& E\left( \frac{ \alpha_{R(p^{(1)}):n}}{\gamma(R(p^{(1)}))} \Big| \epsilon_1=1 \right) = \frac{1}{n}.
\end{eqnarray*}
The method of proof can be used to prove the discussion about least favorable ''false p-values`` given by 
Benjamini and Yekutieli \cite{benjamini_yekutieli} Theorem 5.3.
\end{remark}

\subsection{Proofs of Section \ref{NeuSec4-1}}

% \par\medskip
\noindent
{\sc Proof of Proposition \ref{proposition_29Jan001}.}
By (\ref{proposition_29Jan001Formel001}) we may concentrate on the DU case for the BI model. Note also 
that $g_{a,b}(n_0)$ is strictly increasing in $a$ since the critical values are ordered. On the other 
hand expression (\ref{proposition_29Jan001Formel001}) converges to $\frac{\alpha n}{n+b} = \alpha' < \alpha$ 
for $a \searrow 0$. Thus, the solution $a_1$ of (\ref{FDR_adjustment001}) is unique and it is easy to see that 
$0< a_1 < b$ holds.
% Note first, that uniform upper FDR bounds under $N=n_0$, $1\leq n_0 \leq n$, carry over to random $N$ in Proposition 
% \ref{proposition_29Jan001} and \ref{FurtherFDRAdjustmentProp}. Throughout we deal with this case. 
\hfill $\square$

\medskip
Lemma \ref{LemmaProp3-2ab} leads to a new upper bound for the crucial parameter $a_1$ established in 
Proposition \ref{proposition_29Jan001}.

\begin{proposition}\label{PropositionSharperIneq}
 Consider the assumptions of Proposition \ref{proposition_29Jan001}. Introduce $a_0$ as the unique positive solution of
\begin{equation}
 \alpha = \max_{1\leq n_0 \leq n}\left( \frac{\alpha n_0}{n+b} + \frac{a h(n_0,\alpha')}{n+b} \right),
\end{equation}
where $\alpha' :=\frac{\alpha n}{n+b}$. Then the crucial parameter $a_0$ satisfies $0<a_1<a_0$. 
\end{proposition}

\begin{proof}
Observe first, that the coefficients $\alpha_{i:n} \geq \frac{\alpha' i}{n}$ dominate the BH critical values for the 
choice of $\alpha'$. Thus we have $E_{DU}(V|n_0) \geq h(n_0,\alpha')$ and 
\begin{equation*}
 g_{a,b}(n_0) \geq \frac{\alpha n_0}{n+b} + a \frac{h(n_0,\alpha')}{n+b}.
\end{equation*}
Hence, it is easy to see that the solution $a_1$ of (\ref{FDR_adjustment001}) satisfies $0< a_1 < a_0$.
\end{proof}

\par\bigskip\noindent
{\sc Proof of Proposition \ref{FurtherFDRAdjustmentProp}.}
The requirements (\ref{SecApplications002}) remain true for the new coefficients (\ref{ModificationCritVal}) and we may 
restrict ourselves to the worst DU case. Here we have
\begin{equation*}
 FDR_{DU}(n_0) = E\left( \frac{V}{n-n_0+V} \right).
\end{equation*}
We see that FDR$_{DU}(n_0)$ is ordered by ''$\leq$`` when the critical values and thus the $V$'s are ordered since 
$x\mapsto \frac{x}{n-n_0-x}$ increases. Observe that $k=1$ yield a BH test with FDR $<\alpha$ and the proof is finished.
\hfill $\square$

\subsection{Proofs of Section \ref{NeuSec4-2}}
% \par\bigskip
% \noindent
% {\bf Asymptotics. }
In regular cases the inequalities (\ref{lower_and_upper_FDR_bound}) are asymptotically sharp. For this purpose assume 
that the critical values are generated by a function $h$ via
\begin{equation}\label{Applications001}
 \alpha_{i:n} = h\left( \frac{i}{n} \right), \quad 1 \leq i \leq n.
\end{equation}
At this stage we require that $h:[0,1]\to[0,1]$ is a non-decreasing function with $h(0)=0$, $h(1)<1$ and 
$0 < c_1 = \inf_{0<x<1} \frac{h(x)}{x}, \ \sup_{0<x<1} \frac{h(x)}{x} = c_2 <1$. Then under regularity assumptions 
these bounds (\ref{lower_and_upper_FDR_bound}) are asymptotically sharp in the sense that for $c\in\{c_1,c_2\}$ there 
exist sequences of BI distributions $P_n$ so that 
\begin{equation}
 \lim_{n\to \infty} FDR_{P_n} = \lim_{n\to \infty}\frac{E(N_n)}{n}c.
\end{equation}

\begin{lemma}\label{ApplicationsLemma003}
 Consider the SU test with critical values (\ref{Applications001}) given by a continuous function $h$. Suppose that 
$\frac{R_n}{n} \to K$ converges in probability for some $0<K<1$ and suppose that 
$0 < \lim_{n\to \infty} \frac{E(N_n)}{n} < 1$ exists. Then 
$\lim_{n \to \infty} FDR_n = \frac{h(K)}{K} \lim_{n \to \infty}  \frac{E(N_n)}{n}$.
\end{lemma}

%  Under regularity assumptions the FDR bounds in Example \ref{ApplicationsExample002} are asymptotically sharp. 
Suppose that $c_1 = \frac{h(x_1)}{x_1}$ as well as $c_2 = \frac{h(x_2)}{x_2}$ are attained for some $x_1,x_2 \in (0,1)$. 
Furthermore, assume that there exist sequences of distributions with $\frac{R_n}{n} \to x_1$, $x_2$, in probability, 
respectively, and $0 < \lim_{n\to\infty} \frac{E(N_n)}{n} < 1$. Then Lemma \ref{ApplicationsLemma003} can be applied 
in order to get sharp bounds in (\ref{lower_and_upper_FDR_bound}).

\par\bigskip\noindent
{\sc Proof of Lemma \ref{ApplicationsLemma003}.}
Let $A(K)$ be an open neighborhood of $K$ with $ \frac{h(x)}{x}  \geq d > 0$ on $A(K)$. Then 
\begin{equation*}
 \lim_{n \to \infty} FDR_n = \lim_{n \to \infty} E\left( \frac{V_n}{R_n} 1\left\{ \frac{R_n}{n} \in A(K) \right\} \right) 
\end{equation*}
holds. Introduce 
\begin{equation*}
 0 \leq Z_n := \frac{V_n}{\gamma(R_n)} 1\left\{ \frac{R_n}{n} \in A(K) \right\}  
= \frac{V_n}{n h(\frac{R_n}{n})} 1\left\{ \frac{R_n}{n} \in A(K) \right\} \leq \frac{1}{d}
\end{equation*}
which is a tight sequence of random variables. On the other hand the bounded sequence of random variables 
\begin{equation*}
 W_n := \frac{n h(\frac{R_n}{n})}{R_n} \to \frac{h(K)}{K}
\end{equation*}
converges in probability. Turning to distributional convergent subsequences of $Z_n$ we finally obtain
\begin{equation*}
 \lim_{n \to \infty} E\left( \frac{V_n}{R_n} \right) = \lim_{n \to \infty} E\left( Z_n W_n \right) 
= \frac{h(K)}{K} \lim_{n \to \infty} \frac{E(N_n)}{n}
\end{equation*}
since $ \lim_{n \to \infty} E(Z_n) = \lim_{n \to \infty} \frac{E(N_n)}{n} $ holds by Lemma \ref{LemmaCentral}. \hfill $\square$

\par\bigskip\noindent
{\sc Proof of Theorem \ref{ApplicationsTheorem001}.}
The present proof is given by several steps. Below the following elementary geometric property of the AORC is used. 
Consider for $0<\alpha<1$ the function $A(x,y):=\frac{x}{1-x}\frac{1-y}{y}$ for $(x,y)\in(0,1)^2$ in the plain. Then 
\begin{enumerate}
 \item $A(x,y)=\alpha$ iff $(x,y)$ belongs to the graph of $f_\alpha$, i.e. $f_\alpha(x)=y$. 
 \item $A(x,y)>\alpha$ iff $(x,y)$ is below the graph of $f_\alpha$, i.e. $f_\alpha(x)>y$. 
 \item $A(x,y)<\alpha$ iff $(x,y)$ lies above the graph of $f_\alpha$, i.e. $f_\alpha(x)<y$. 
\end{enumerate}

\noindent
{\bf (I)} We claim $0<\beta<1$. We first show $\beta>0$. Therefore, choose $n_0 = \lfloor \frac{n}{2} \rfloor$ and a Dirac uniform 
$DU(n,n_0)$ configuration. Then $\frac{R_n}{n} \geq \frac{1}{2}$ holds. Hence every $p$-value $p_i \leq f^{-1}(\frac{R_n}{n})$ will 
be rejected and in particular every $p$-value $p_i \leq f^{-1}(\frac{1}{2})$. Thus 
\begin{equation*}
 \beta \geq \liminf_{n\to\infty} FDR_{DU(n,n_0)} \geq \liminf_{n\to\infty} E\left( \frac{V_n}{n} \right) 
\geq \frac{1}{2}f^{-1}\left(\frac{1}{2}\right) > 0
\end{equation*}
for this sequence of Dirac uniform models. Let us next show $\beta<1$. Therefore, let $n_0$ be arbitrary and observe that $f$ lies 
above the Benjamini Hochberg rejection curve $f_{BH}(x)=(1+\epsilon)x$ of the BH test. By Lemma \ref{ApplicationsLemma001} (a)(i) 
we always have 
$$ FDR_{P_n} \leq \frac{1}{1+\epsilon}. 
$$ 
% that $\alpha_{n:n} = f^{-1}(1) < 1$ holds. Because of the concavity of $f$ (\ref{SecApplications002}) holds and hence 
% $\alpha_{j:n} \leq \frac{j}{n}f^{-1}(1)$follows. From Lemma \ref{ApplicationsLemma001} (a)(i) we deduce that the FDR 
% for all $n$ and all considered distributions is simultaneously bounded by $f^{-1}(1)<1$. 

\noindent
{\bf (II)} The statement (\ref{beta_ApplicationsTheorem002}) of the Theorem is first proved for concave rejection curves $f$. 
Therefore, observe that under the distribution $P_n^{N_n}$ of $N_n$
\begin{equation*}
 FDR_{P_n} = \int FDR_{P_n}(n_0) P_n^{N_n}(dn_0) \leq \sup_{n_0 \leq n} FDR_{P_n}(n_0) 
\leq \sup_{n_0 \leq n} FDR_{DU(n,n_o)}
\end{equation*}
for all $P_n \in \mathcal{P}_n$, since the Dirac uniform configuration is least favorable for the FDR for SU tests 
with critical values fulfilling (\ref{SecApplications002}). Hence we get 
\begin{equation}\label{EqToShow001}
 \beta = \limsup_{n\to\infty} \sup_{n_0 \leq n} FDR_{DU(n,n_0)} 
\end{equation}
since $DU(n,n_0)$ belongs to $\mathcal{P}_n$ and thus there exists a subsequence $(n,n_0)=(n_n,n_{0,n})_n$, again denoted 
by $n$, with 
\begin{equation}\label{Subsequence001t}
 FDR_{DU(n,n_0)} \to \beta \quad \mbox{and} \quad \frac{n_0}{n} \to 1-y \quad \mbox{for } n\to\infty
\end{equation}
and some $1-y \in [0,1]$. Now we determine the limit of $FDR_{DU(n,n_0)}$ for every sequence $\frac{n_0}{n}\to1-y\in[0,1]$. 

1. For $\frac{n_0}{n}\to0$ observe that 
\begin{equation*}
 FDR_{DU(n,n_0)}= E\left( \frac{V_n}{n-n_0+V_n} \right) \leq \frac{n_0}{n} \underset{n\to\infty}{\longrightarrow} 0
\end{equation*}
and by a similar argument it follows that the limit of $FDR_{DU(n,n_0)}$ is continuous in $1-y$ at $0$.  
% Furthermore, we will consider the subsequence and  

2. Let us consider $\frac{n_0}{n}\to 1-y \in (0,1)$ with positive $y$ and introduce the straight line 
$g(t)=y+(1-y)t$ which runs through the points $(0,y)$ and $(1,1)$ and has the unique crossing point $(x,K)$, $0<K<1$, 
with $f$. 
% Then $\frac{R_n}{n} \to K \ a.s.$ holds. To see this 
Observe that 
\begin{equation*}
 \hat F_n(\alpha_{R_n:n}) = f(\alpha_{R_n:n})
\end{equation*}
holds. Now let $Z$ be a weak accumulation point of $\alpha_{R_n:n}$. Since $f$ is continuous and $\hat F_n$ 
converges uniformly to $g$ with probability 1 the equation
\begin{equation*}
 g(Z) = y + (1-y) Z = f(Z)
\end{equation*}
follows. There is only one crossing point and thus $Z=x$ is constant for each weak accumulation point. 
From $f^{-1}(\frac{R_n}{n}) = \alpha_{R_n:n}$ we now deduce 
\begin{equation}\label{ApplicationTheoremProof001}
 \frac{R_n}{n} = f(\alpha_{R_n:n}) \to f(x) = K \quad a.s.
\end{equation}
at least along subsequences. 
% By Lemma 2.9 of Scheer \cite{scheer} it actually follows that $\alpha_{R_n:n}$ converges almost surely to the crossing point $x$. 
Similar arguments were used by Scheer \cite{scheer}, Lemma 2.9, in his set up in order to prove that $\alpha_{R:n}$ 
converges to the crossing point $x$. 

A simple geometric argument for the gradient of $g$ yields 
\begin{equation*}
 \frac{1-y}{1} = \frac{f(x)-y}{x}
\end{equation*}
and hence 
\begin{equation}\label{ApplicationTheoremProof002}
 y= \frac{f(x)-x}{1-x}, \quad 1-y=\frac{1-f(x)}{1-x}.
\end{equation}
By Lemma \ref{ApplicationsLemma003} and subsequence arguments 
\begin{equation}\label{ApplicationTheoremProof002Darst1}
 \lim_{n\to\infty} FDR_{DU(n,n_0)} = \frac{f^{-1}(K)}{K} (1-y) = \frac{x}{f(x)} \frac{1-f(x)}{1-x} 
\end{equation}
holds for all sequences $\frac{n_0}{n} \to 1-y \in (0,1)$. 
% In particular, for the subsequence in (\ref{Subsequence001t}) $\beta = \frac{x}{f(x)} \frac{1-f(x)}{1-x}$ holds.

% Along a suitable subsequence we finally get 
% \begin{equation*}
%  \beta = \lim_{n\to\infty} FDR_{DU(n,n_0)} = \frac{f^{-1}(K)}{K} (1-y) = \frac{x}{f(x)} \frac{1-f(x)}{1-x}.
% \end{equation*}
% by appliying Lemma \ref{ApplicationsLemma003}.
3. Now consider $\frac{n_0}{n} \to 1$. Again $\frac{R_n}{n} \to 0$ in distribution. This follows from (\ref{ApplicationTheoremProof001}) 
by the monotonicity of $R_n$ in $n_0$ since $K\to 0$ holds for $y\to 0$. Observe next that for every $x>0$ we have 
$z:=f(x)>0$ and hence by (\ref{SecApplications002})
\begin{equation*}
 \frac{n\alpha_{j:n}}{j} \leq \frac{n\alpha_{\lfloor nz \rfloor:n}}{\lfloor nz \rfloor} 
= \frac{n f^{-1}(\frac{\lfloor nz \rfloor}{n})}{\lfloor nz \rfloor} < 1
\end{equation*}
holds for $j\leq \lfloor nz \rfloor$. By Lemma \ref{ApplicationsLemma001} (a)(i) 
\begin{equation*}
 FDR_{DU(n,n_0)} \leq \frac{h(\frac{\lfloor nz \rfloor}{n})}{\frac{\lfloor nz \rfloor}{n}} \frac{n_0}{n} 
+ P(R_n > \lfloor nz \rfloor ) \to \frac{f^{-1}(z)}{z} = \frac{x}{f(x)}
\end{equation*}
holds and hence 
\begin{equation*}
 \limsup_{n\to\infty} FDR_{DU(n,n_0)} \leq \lim_{x\searrow 0} \frac{x}{f(x)} 
= \lim_{x\searrow 0} \frac{x}{f(x)} \frac{1-f(x)}{1-x}
\end{equation*}
when $\frac{n_0}{n}\to 1$ since $\lim_{x\searrow 0} \frac{1-f(x)}{1-x} =1$.
% \par\bigskip
% Zeige hier nun f\"ur $\frac{n_0}{n}\to 1$ dass 
% \begin{equation*}
%  \limsup \sup FDR \leq \lim_{x \searrow 0} \frac{x}{f(x)}\frac{1-f(x)}{1-x}.
% \end{equation*}

Altogether, by subsequence arguments and (\ref{EqToShow001}) we directly obtain (\ref{beta_ApplicationsTheorem002})
% Altogether we get 
% \begin{eqnarray*}
%  H(1-\beta) \! \! &=& \! \! \sup\left\{ \frac{x}{1-x} \frac{1-f(x)}{f(x)} \, : \, 0<x \mbox{ and } \frac{f(x)-x}{1-x}\leq 1-\beta \right\} \\
% &=& \! \! \sup\left\{ \lim_{n\to\infty} FDR_{DU(n,n_0)} \, : \, \frac{n_0}{n} \to 1-y \mbox{ with } 1-y \geq \beta \right\} \\
%  &\leq& \! \! \beta \leq H(1-\beta)
% \end{eqnarray*}
% where the first equality holds by definition, the second equality follows from (\ref{ApplicationTheoremProof002}) 
% and (\ref{ApplicationTheoremProof002Darst1}) and the first inequality directly holds by the definition of $\beta$. 
% The last inequality follows by considering the appropriate sequence $(n,n_0)_n$ for which 
% $\lim_{n\to\infty} FDR_{DU(n,n_0)} = \beta$ holds and by applying the previous equalities of the formula again.  

\noindent
{\bf (III)} Let now $f$ be the general rejection curve of Theorem \ref{ApplicationsTheorem001}. Introduce \\
$\gamma:= \sup\{ \frac{x}{1-x}\frac{1-f(x)}{f(x)} : 0\leq x \leq x_0 \}$. 

(a) Claim: $\beta \leq \gamma$. The geometric arguments 1.-3. at the beginning of the proof imply $f\geq f_\gamma$ 
for the AORC with parameter $\gamma$. Next $f_\gamma$ is modified as follows. Let $x\mapsto r(x)$ be the tangent 
straight line attached at $f_\gamma$ at the point $(x_0,f_\gamma(x_0))$. Then 
\begin{equation}
 \tilde f (x) = \begin{cases}
                 f_\gamma(x) & x\leq x_0 \\
		 r(x) & x>x_0
                \end{cases}
\end{equation}
defines a concave rejection curve $\tilde f \leq f$ with $\tilde f(x_1)=1$ for some $x_0 < x_1 < 1$. 
By Lemma \ref{BeliebigeAblehnkurve}, given below, the $FDR(\tilde f)$ is always an upper bound of the 
$FDR(f)$. By step (II) the asymptotic worst case FDR of $FDR(\tilde f)$ equals 
\begin{equation}
 \sup\left\{ \frac{x}{1-x}\frac{1-\tilde f(x)}{\tilde f} : 0 \leq x \leq x_1 \right\} \geq \beta.
\end{equation}
It is easy to see that the left hand side is just $\gamma$. 

(b) The proof will be completed by showing $\beta \geq \gamma$. The inequality follows from the next special 
construction of mixture models, where now in contrast to part (II) of this proof the DU configuration is no 
longer least favorable. Introduce for each $0<\delta<1$ the straight line $g_\delta(t) = 1-\delta +t\delta$ 
and the intersection set $D_\delta := \{x\in(0,x_0) : g_\delta(x)=f(x)\}$. Note that $D_\delta$ is a compact 
set bounded away from $0$ and $x_0$ with $D_\delta \neq \emptyset$ and
\begin{equation}
 \bigcup_{0<\delta<1} D_\delta = (0,x_0).
\end{equation}
Thus, there exists some maximal element $x_\delta \in D_\delta$ with $x\leq x_\delta$ for all $x\in D_\delta$. 
In the next step we will introduce for each set $D_\delta$ a mixture model with appropriate distribution function 
\begin{equation}
 F_\delta(t) = \delta t + (1-\delta) G_\delta(t). 
\end{equation}
The non-uniform part $G_\delta$ will have the following properties:
\begin{itemize}
 \item[(i)] $G_\delta(x_\delta)=1$, which implies $F_\delta(t) = g_\delta(t)$ for all $t\geq x_\delta$ and 
		$F_\delta(x_\delta) = f(x_\delta)$. 
 \item[(ii)] $F_\delta(t) < f(t)$ for all $t< x_\delta$. 
\end{itemize}
In order to do so, let $h$ be a straigt line through $(x_\delta,f(x_\delta))$ with sufficiently large slope 
such that $h(t)<f(t)$ for all $0<t<x_\delta$. Here the left-sided differentiability of $f$ is used. Put now 
\begin{equation}
 F_\delta(t) = \max(\delta t, h(t)), \quad \mbox{for } t\leq x_\delta.
\end{equation}
This $x_\delta$ is the only cut point of $F_\delta$ and $f$. Consider now a mixture model $P_n$ with distribution 
function $f_\delta$. Similarly as in (II) we have $\frac{N_n}{n}\to\delta$ and via the cut point consideration 
we arrive at 
\begin{equation}
 FDR_{P_n} \longrightarrow \frac{x_\delta}{1-x_\delta} \frac{1-f(x_\delta)}{f(x_\delta)} =: \beta_\delta
\end{equation}
at least along suitable subsequences. A comparison of the line segment $\{(x,g_\delta(x)) : x \in D_\delta\}$ 
with the AORC $f_{\beta_\delta}$ yields $\frac{x}{1-x} \frac{1-f(x)}{f(x)} \leq \beta_\delta$ for each 
$x\in D_\delta$ if we take 1.-3. into account. This construction can be done for each set $D_\delta$. Thus, 
the proof of the inequality $\beta\geq\gamma$ is complete. 
\hfill $\square$

\par\bigskip
The next lemma is used in the last proof and may be of separate interest.

\begin{lemma}\label{BeliebigeAblehnkurve}
 Let $f:[0,1]\to \mathbb{R}$ be a non-decreasing rejection curve with $f(0)=0$, $f(x_0)=1$ for some $x_0<1$ and 
$f(x)\geq (1+\epsilon)x$ for some $\epsilon>0$. Moreover, let $r_0:[0,1]\to\mathbb{R}$ be a concave rejection 
curve and a lower bound of $f$. Under the BI Model for fixed $n$ and $N=n_0$ we have 
\begin{equation}
 FDR(f) \leq FDR(r_0) \leq FDR_{DU(n,n_0)}(r_0)
\end{equation}
for the FDR of the SU tests based on $f$ and $r_0$ via (\ref{critVal_rejection_curve_f}). 
\end{lemma}

\begin{proof}
 Let $\alpha_{i:n}$ and $\alpha_{i:n}^{(0)}$ be the critical values of the rejection curves $f$ and $r_0$, respectively, 
defined by (\ref{critVal_rejection_curve_f}). Define $R$ and $R_0$ as the number of rejections of the SU test based on 
$\alpha_{i:n}$ and $\alpha_{i:n}^{(0)}$, respectively. It is easy to see that $\alpha_{i:n} \leq \alpha_{i:n}^{(0)}$ and 
hence $R \leq R_0$ almost surely holds. Using the technique of Remark \ref{RemarkzumCentralLemmaProof} we obtain 
\begin{eqnarray*}
FDR(f) &=& n_0 E\left( \frac{\alpha_{R(p^{(1)}):n}}{R(p^{(1)})} \Big| \epsilon_1=1 \right) 
\leq n_0 E\left( \frac{\alpha^{(0)}_{R(p^{(1)}):n}}{R(p^{(1)})} \Big| \epsilon_1=1 \right) \\
&\leq& n_0 E\left( \frac{\alpha^{(0)}_{R_0(p^{(1)}):n}}{R_0(p^{(1)})} \Big| \epsilon_1=1 \right) 
\leq n_0 E_{DU(n,n_0)}\left( \frac{\alpha^{(0)}_{R_0(p^{(1)}):n}}{R_0(p^{(1)})} \Big| \epsilon_1=1 \right) \\
&=& FDR_{DU(n,n_0)}(r_0). 
\end{eqnarray*}
since $i\mapsto\frac{\alpha_{i:n}^{(0)}}{i}$ is non-decreasing and the $DU(n,n_0)$ configuration is least favorable for 
the FDR, cf. Benjamini and Yekutieli \cite{benjamini_yekutieli}.
\end{proof}

\begin{remark} 
Consider the BI Model and let the false $p$-values be independent and stochastically smaller than the uniform 
distribution. Let $FDR_{U(0,1)}$ denote the FDR under uniformly distributed false $p$-values. If we replace 
$r_0$ in Lemma \ref{BeliebigeAblehnkurve} by a convex rejection curve $r_0$ which is a upper bound of $f$, it 
is easy to see by similar arguments that 
\begin{equation}
 FDR(f) \geq FDR_{U(0,1)}(r_0)
\end{equation}
holds for the FDR of the SU tests based on $f$ and $r_0$, since $i\mapsto\frac{\alpha_{i:n}^{(0)}}{i}$ is non-increasing 
for the critical values $\alpha_{i:n}^{(0)}$ corresponding to the rejection curve $r_0$. According to Benjamini and 
Yekutieli \cite{benjamini_yekutieli}, here the uniform distribution is least favorable for the FDR under stochastically 
smaller $p$-values.  
\end{remark}

\medskip 
If (\ref{PRDS}) is non-increasing, then the $p$-values are called to be negative regression dependent on the subset 
of true null hypotheses (NRDS). Under this assumption the lower bound in (\ref{lower_and_upper_FDR_bound}) and Lemma 
\ref{ApplicationsLemma001} (b)(i) stay true, see Heesen \cite{heesen} for instance.

\par\bigskip\noindent
{\sc Proof of Theorem \ref{ApplicationsTheoremSD}.} 
By Gontscharuk \cite[Theorem 3.10]{gontscharuk} it is well known that 
\begin{equation*}
 E\left( \frac{V_{SD}}{R_{SD}}\Big| n_0, \bar f \right) \leq E\left( \frac{V_{SU}}{R_{SU}}\ \Big| n_0, \bar f \right)
\end{equation*}
holds for tests with critical values (\ref{SecApplications002}), see also Heesen \cite{heesen} Lemma 2.29. The inequality 
also follows from the technique used in Remark \ref{RemarkzumCentralLemmaProof}. 
% First observe that $R_{SD} \leq R_{SU}$ holds and that $i \mapsto \frac{\alpha_{i:n}}{i}$ is non-decreasing since 
% $f$ is a concave rejection curve. 
% Thus conditioned under $N=n_0$ and the vector of false $p$-values $\bar f$ we have 
% \begin{eqnarray*}
%  E\left( \frac{V_{SD}}{R_{SD}}\Big| n_0, \bar f \right)
% &=& n_0 E\left( \frac{\alpha_{R_{SD}:n}}{R_{SD}}\Big| n_0, \bar f \right) 
% \leq n_0 E\left( \frac{\alpha_{R_{SU}:n}}{R_{SU}}\Big| n_0, \bar f \right) \\
% &=& E\left( \frac{V_{SU}}{R_{SU}}\ \Big| n_0, \bar f \right)
% \end{eqnarray*}
% and
 Hence 
\begin{equation}\label{ApplicationsTheoremSDProofgeaendert001}
 \limsup_{n\to \infty} \sup_{P_n \in \mathcal{P}_n} FDR_{P_n,SD} 
\leq \limsup_{n\to \infty} \sup_{P_n \in \mathcal{P}_n} FDR_{P_n}=:\beta.
\end{equation} 
In the proof of Theorem \ref{ApplicationsTheorem001} we already showed that $\beta$ is 
attained by $\frac{x}{1-x}\frac{1-f(x)}{f(x)}$ for some $x$. Thus, to show ''$=$'' on 
the left hand side in (\ref{ApplicationsTheoremSDProofgeaendert001}) it suffices to show 
that for all such $x$ there is some sequence of distributions $(P_n)_n$ so that $FDR_{P_n,SD}$ 
converges to $\frac{x}{1-x}\frac{1-f(x)}{f(x)}$. Therefore, let us now consider sequences 
of $DU(n,n_0)$ configurations with $\frac{n_0}{n}\to 1-y \in (0,1)$. Along the lines 
of the proof of Theorem \ref{ApplicationsTheorem001} we have 
\begin{equation}
 \frac{R_{SD,n}}{n} \to K \quad a.s. \ \mbox{with } 0<K<1.
\end{equation}
Observe that $0<K$ holds since $\frac{R_{SD,n}}{n} \geq \frac{n-n_0}{n} \to y >0$. Moreover, $K<1$ holds since $f(x)=1$ for all 
$x_0 \leq x$ for some $x_0<1$. Now with $x, f(x)$ and $y$ as in (\ref{ApplicationTheoremProof001}) and 
(\ref{ApplicationTheoremProof002}) in the proof of Theorem \ref{ApplicationsTheorem001} we have
\begin{eqnarray*}
FDR_{DU(n,n_0),SD} &=& E_{DU(n,n_0)}\left( \frac{V_{SD,n}}{R_{SD,n}} \right) 
= E_{DU(n,n_0)}\left( \frac{\frac{R_{SD,n}}{n} - \frac{n-n_0}{n}}{\frac{R_{SD,n}}{n}} \right) \\
&\underset{n\to\infty}{\longrightarrow}& \frac{K-y}{K} = \frac{x}{f(x)}\frac{1-f(x)}{1-x}
\end{eqnarray*}
by dominated convergence. Hence we have ``='' in (\ref{ApplicationsTheoremSDProofgeaendert001}) since the above formula holds for 
all $1-y \in (0,1)$ and the representation (\ref{beta_ApplicationsTheorem002}).
\hfill $\square$

\subsection{Proofs of Section \ref{SecAdaptiveControl}}

% \par\bigskip
{\sc Proof of Theorem \ref{AdaptiveControlTheorem1}.} 
First observe that $FDR_{P_n}= 0 \leq \alpha$ obviously holds if $N_n=0$. Thus, without loss of generality 
let $N_n>0$ almost surely for all $n$. Conditioned under $\epsilon=(\epsilon_1,\ldots, \epsilon_n)$ we obtain
by Lemma \ref{LemmaCentral} (a) and (\ref{CritVal002}) 
\begin{equation}
  \frac{N_n}{n} = E\left( \frac{V}{n\hat \alpha_{R:n}} \Big| \epsilon \right) \geq E\left( \frac{\hat n_{0,n}}{n}\frac{V}{R\alpha} \Big| \epsilon \right) 
\end{equation}
for the reverse martingale model since the conditional case is also included. Thus by integration
% by (\ref{AdaptiveControlTheorem1003}) 
\begin{eqnarray*}
 \alpha &\geq& E_{P_n}\left( \frac{\hat n_{0,n}}{N_n} \frac{V}{R} \right) 
 \geq E_{P_n}\left( (1-\delta) \frac{V}{R} 1\left\{ \frac{\hat n_{0,n}}{N_n} > 1-\delta \right\} \right) \\
&=& (1-\delta) E_{P_n} \left( \frac{V}{R} \right) - (1-\delta) E_{P_n}\left( \frac{V}{R} 1\left\{ \frac{\hat n_{0,n}}{N_n} \leq 1-\delta \right\} \right) \\
&\geq& (1-\delta) E_{P_n} \left( \frac{V}{R} \right) -(1-\delta) P_n\left( \frac{\hat n_{0,n}}{N_n} \leq 1-\delta \right)
\end{eqnarray*}
holds for every $\delta > 0$ and the statement follows by (\ref{AdaptiveControlTheorem1001}). 
\hfill $\square$

\par\bigskip\noindent
{\sc Proof of Proposition \ref{Proposition002}.} 
Introduce the set $A_n=\{ \hat\alpha_{R:n} < \lambda \}$. 
% and its complement $A_n^c$. Observe that in the proof of 
% Lemma \ref{Theorem003} $\alpha'=\frac{n\hat F_n(\lambda)}{\lambda \hat n_0}\alpha \leq 1$ holds on $A_n$. 
% Else if we assume $\alpha'>1$, then in the notation of the proof $R=R_q=n(\lambda)$ follows and  
% $A_n=\{ \frac{n(\lambda)}{\hat n_0}\alpha < \lambda \}$ contradicts the assumption. Hence by the arguments of  
% the proof of Lemma \ref{Theorem003} we obtain
An inspection of the proof of Lemma \ref{Theorem003} (given below) yields 
\begin{equation*}
 E_{P_n}\left( \frac{V_n}{R_n}1_{A_n} \right) = \frac{\alpha}{\lambda}E_{P_n}\left( \frac{V_n(\lambda)}{\hat n_0}1_{A_n} \right)
\end{equation*}
 restricted on $A_n$. Define $S_n(\lambda) := \sum_{i=1}^n (1-\epsilon_i) 1\{p_i \leq \lambda\}$, then
\begin{equation*}
 \alpha \geq \limsup_{n\to\infty} \frac{\alpha(1-\lambda)}{\lambda} E_{P_n}\left( \frac{V_n(\lambda)}{n-V_n(\lambda)-S_n(\lambda)+\kappa_n} \right)
\end{equation*}
follows. First, let us consider case (ii). For each $\delta>0$ we have 
\begin{equation}\label{Proposition002Proof001}
 1 \geq \frac{(1-\lambda)}{\lambda} \limsup_{n\to\infty}   E_{P_n}\left( \frac{V_n(\lambda)/N_n}{1- V_n(\lambda)/N_n +\kappa_n/N_n + \delta}1_{A_n} \right).
\end{equation}
Observe next that $0 \leq V_n(\lambda)/N_n \leq 1$ is tight and we consider an arbitrary distributional cluster point $Z$ of $V_n(\lambda)/N_n$. 
The appertaining subsequence is for convenience also denoted by $\{n\}$, i.e. $V_n(\lambda)/N_n \to Z$ in distribution. 
Note that $E(Z) = \lambda$ and 
\begin{equation*}
 1 \geq \frac{1-\lambda}{\lambda} E\left(\frac{Z}{1-Z+\delta}\right)
\end{equation*}
hold. Now Jensen's inequality implies 
\begin{equation*}
 \frac{\lambda}{1-\lambda} \geq E\left( \frac{Z}{1-Z} \right) \geq \frac{\lambda}{1-\lambda}
\end{equation*}
when $\delta \searrow 0$. Since $x \mapsto \frac{x}{1-x}$ is strictly convex we have $Z=\lambda$ a.e. Since $Z$ was 
an arbitrary cluster point we conclude $V_n(\lambda)/N_n \to \lambda$ in $P_n$-probability. This statement implies 
the result
\begin{equation*}
 \frac{\hat n_0}{N_n} = \frac{N_n - V_n(\lambda) +n\kappa_n}{(1-\lambda) N_n} \to 1.
\end{equation*}
In case (i) the proof is similar. Note that the assumption then implies 
\begin{equation*}
 \frac{(n-N_n)-S_n(\lambda)}{N_n} \to 0
\end{equation*}
and we may proceed as in (\ref{Proposition002Proof001})
\hfill $\square$

\par\bigskip\noindent
{\sc Proof of Remark \ref{RemarkAsymptoticVariability}.} 
Observe that 
\begin{eqnarray*}
 \frac{\hat n_0}{N_n} &\geq& \frac{N_n - \sum_{i=1}^n \epsilon_i 1\{ p_i \leq \lambda \} + \kappa_n}{N_n (1-\lambda)} \\
&\geq& \frac{1 - \frac{1}{N_n}\sum_{i=1}^n \epsilon_i 1\{ p_i \leq \lambda \} }{1-\lambda} \to 1 
\end{eqnarray*}
holds, where the right hand side converges by (\ref{AboutAsymptoticFDR}) in probability w.r.t. the conditional convergence.
\hfill $\square$

\par\bigskip\noindent
{\sc Proof of Theorem \ref{AdaptiveControlTheorem1Part2}.} 
Conditioned under $\epsilon=(\epsilon_1,\ldots, \epsilon_n)$ we directly obtain 
by Lemma \ref{LemmaCentral} (b) and (\ref{AdaptiveControlTheorem1002}) that
\begin{equation}
 \frac{N_n}{n} \geq E\left( \frac{\hat n_{0,n}}{n}\frac{V}{R\alpha} \Big| \epsilon \right) 
\end{equation}
holds and the statement follows by the same arguments as in the proof of Theorem \ref{AdaptiveControlTheorem1}.
\hfill $\square$

\par\bigskip\noindent
{\sc Proof of Lemma \ref{Theorem003}.} 
Conditioned under $(\epsilon_i)_{i\leq n}=\epsilon$, $(p_i\, : \, \epsilon_i=0)=\bar f$ and $(n\hat F_n(t))_{t\geq \lambda})= (n(t))_{t\geq \lambda}$ 
we have exactly $V(\lambda)$ true $p$-values smaller or equal to $\lambda$ where $V(\lambda)$ is a fixed number. 
Without restriction we assume $N \geq V(\lambda)>0$ since everything is obviously fine for the excluded cases.
Let us now consider new rescaled $p$-values $q_i$, $i=1,\ldots, n(\lambda)$ defined by 
\begin{equation*}
 q_{i:n(\lambda)} := \frac{p_{i:n}}{\lambda}, \quad i=1,\ldots,n(\lambda).
\end{equation*}
When a new $p$-values $q_i$ corresponds to a true null hypothesis it is again uniformly distributed on $(0,1)$ and 
$ \frac{1\{q_i\leq t\}}{t} $ is a reverse martingale with respect to the reverse filtration 
$\mathcal{F}_t^q = \sigma(1\{q_j\leq s\} \, : \, 1\leq j \leq n(\lambda), s\geq t)$ under the above conditional assumption.
The exact positions of the $V(\lambda)$ true $p$-values in $(q_1,\ldots, q_{n(\lambda)})$ does not matter for our considerations. 
We now apply Lemma \ref{LemmaCentral} (a) for the SU multiple test with critical values 
\begin{equation*}
 \alpha_{i:n(\lambda)}^{(q)} := \frac{i}{n(\lambda)}\alpha' \quad 
\mbox{and } \alpha':= \frac{n(\lambda)}{\lambda \hat n_0} \alpha
\end{equation*}
on the $q$'s. The data dependent level $\alpha'$ only depends by assumption (A1) on the information 
given by $\mathcal{F}_{\lambda}$. Conditionally under $\mathcal{F}_{\lambda}$ we have a regular non data dependent 
SU procedure on the $q$'s. Let $R_q$ and $V_q$ denote the number of rejections and false rejections respectively 
by the above SU test. Observe that 
\begin{equation}\label{Theorem003Bew001}
 E\left( \frac{V_q}{R_q} \Big|\epsilon, \bar f, (n(t))_{t\geq \lambda}  \right) 
= \frac{V(\lambda)}{n(\lambda)} \min(\alpha',1) 
%= \frac{V(\lambda)}{\lambda \hat n_0}\alpha
\end{equation}
holds by Lemma \ref{LemmaCentral} (a). Obviously (\ref{Theorem003Bew001}) is $\frac{V(\lambda)}{n(\lambda)}$ in case $\alpha' \geq 1$.

Now observe that 
\begin{eqnarray*}
 R_q &=& \max\{i\leq n(\lambda) \, : \, q_{i:n(\lambda)} \leq \alpha_{i:n(\lambda)}^{(q)} \} \\
&=& \max\left\{i\leq n(\lambda) \, : \, \frac{p_{i:n}}{\lambda} \leq  \frac{i}{n(\lambda)} \frac{n(\lambda)}{\lambda \hat n_0} \alpha \right\} \\
&=& \max\left\{ i\leq n \, : \, p_{i:n}\leq  \left(\frac{i}{\hat n_0} \alpha\right) \wedge \lambda \right\} =R 
\end{eqnarray*}
and hence $V_q=V$ since both tests, belonging to $R$ and $R_q$, are rejecting the same hypotheses. 
Thus by (\ref{Theorem003Bew001}) we get
\begin{eqnarray*}
   E\left( \frac{V}{R} \right) 
&=& E\left( E\left( \frac{V}{R}\, \Big|\epsilon, \bar f, (n(t))_{t\geq \lambda}  \right) \right)\\
&=& E\left( E\left( \frac{V_q}{R_q} \, \Big| \epsilon, \bar f, (n(t))_{t\geq \lambda}  \right) \right)\\
&=& E\left( \frac{V(\lambda)}{n(\lambda)}\min(\alpha',1) \right) 
 =\frac{\alpha}{\lambda} E\left( V(\lambda) \min\left\{\frac{1}{\hat n_0}, \frac{\lambda}{n\hat F_n(\lambda) \alpha}\right\} \right).
\hfill \quad \quad \square
\end{eqnarray*}

\par\bigskip\noindent
{\sc Proof of Proposition \ref{Proposition001}.} 
Choose $N=n$ and $p_i=U$ for all $i\leq n$. Thus $V(\lambda)=n1\{U\leq\lambda\}$ and $n\hat F_n(\lambda)=n1\{U\leq\lambda\}$.
The exact FDR formula (\ref{OurCondition001}) yields 
\begin{equation*}
 \alpha \geq FDR = \frac{\alpha}{\lambda} n\lambda \min\left( \frac{1}{\hat n_0(\vec 1)} ,\frac{\lambda}{n\alpha}  \right)
\end{equation*}
where $\hat n_0(\vec 1)$ stands for the value of the estimator when $\hat F_n(\rho) = 1$ for all $\rho\geq \lambda$. Thus 
$\frac{n}{\hat n_0(\vec 1)} \leq 1$ holds. By our assumption we have $\hat n_0 \geq \hat n_0(\vec 1)$. 
\hfill $\square$

\par\bigskip\noindent
We need the following ``balayage`` lemma for the proof of Theorem \ref{Theorem004}. 
\begin{lemma}\label{LemmaforTheorem004}
 Let $f:\{0,1,\ldots, m\} \to \mathbb{R}$ be a convex function and let $P=\sum_{j=1}^m p_j \epsilon_j$ be a distribution on 
$\{0,1,\ldots, m\}$, where $\epsilon_j$ denotes the Dirac distribution on $\{j\}$. Then $E_P(id) = E_{P'}(id)$ and 
$E_P(f) \leq E_{P'}(f)$ hold for the distribution 
$$P'= \left(1-\sum_{j=1}^m \frac{j}{m}p_j\right)\epsilon_0 + \sum_{j=1}^m \frac{j}{m}p_j \epsilon_m.$$
\end{lemma}

\begin{proof}
$E_P(id) = E_{P'}(id)$ obviously holds by the definition of $P'$. 
% First introduce $p_j^m =\frac{j}{m}p_j$ and $p_j^0 =\frac{m-j}{m}p_j$ for $j=0,\ldots, m$. 
Then by the convexity of $f$ we have 
\begin{eqnarray*}
 E(f) &=& \sum_{j=0}^m f(j)p_j \leq \sum_{j=0}^m \left( \frac{m-j}{m} f(0) p_j + \frac{j}{m}f(m)p_j \right)\\
&=& \left(1-\sum_{j=1}^m \frac{j}{m}p_j\right) f(0) + \sum_{j=1}^m \frac{j}{m}p_j f(m) = E_{P'}(f).
%  \sum_{j=0}^m (p_j^0 f(0) + p_j^m f(m) ) = 
\end{eqnarray*}
\end{proof}

\par\bigskip\noindent
{\sc Proof of Theorem \ref{Theorem004}.} 
Throughout the condition (\ref{verification}) will be verified. For the portion $\widetilde G_i$ of true $p$-values 
introduce the quantities
\begin{equation}
 V_i(\lambda) :=\sum_{j\, : \, p_j \in \widetilde G_i} 1\{p_j \leq \lambda\}, \quad N_i := |\widetilde G_i|
\end{equation}
with $V(\lambda)= \sum_{i=1}^k V_i(\lambda)$. Whenever $\widetilde G_i \neq \emptyset$ holds select one true $p$-value 
$\widetilde p_i \in \widetilde G_i$. Observe that the $\widetilde  p_i$ are conditionally independent given $\epsilon$. 
Under that condition we have for (\ref{Grouping002})
\begin{eqnarray*}
 E\left( \frac{V(\lambda)}{\hat n_0(\kappa)} \Big| \epsilon \right) 
&=& (1-\lambda) E\left( \frac{\sum_{i=1}^k V_i(\lambda)}{n-n\hat F_n(\lambda) +\kappa} \Big| \epsilon \right) \\
&\leq& (1-\lambda) E\left( \frac{\sum_{i=1}^k V_i(\lambda)}{N-\sum_{i=1}^k V_i(\lambda)+\kappa} \Big| \epsilon \right)=(\star). 
% &=&  (1-\lambda) E\left( \frac{\sum_{i=1}^k V_i(\lambda)}{mk-\sum_{i=1}^k V_i(\lambda)+m} \Big| \epsilon \right) 
\end{eqnarray*}
In the next step we are going to condition under $a=\sum_{i=2}^k V_i(\lambda)$. Note that then $V_1(\lambda)$ is a 
$B(N_1,\lambda)$ distributed random variable. By Lemma \ref{LemmaforTheorem004}, $V_1(\lambda)$ can be substituted 
by the worst case random variable $N_1 1\{\widetilde p_1 \leq \lambda\}$, i.e. 
\begin{equation*}
 E\left( \frac{ V_1(\lambda)+a}{N-V_1(\lambda)-a+\kappa} \Big| \epsilon,a \right) 
\leq E\left( \frac{ N_1 1\{\widetilde p_1 \leq \lambda\}+a}{N-N_1 1\{\widetilde p_1 \leq \lambda\}-a+\kappa} \Big| \epsilon,a \right).
\end{equation*}
If we proceed $k$ times we arrive at the upper bound 
\begin{eqnarray*}
(\star) &\leq& (1-\lambda) E\left( \frac{\sum_{i=1}^k N_i 1\{\widetilde p_i \leq \lambda\}}{N-\sum_{i=1}^k N_i 1\{\widetilde p_i \leq \lambda\}+\kappa} \Big| \epsilon \right)\\
&\leq& (1-\lambda) E\left( \frac{\sum_{i=1}^k m 1\{\widetilde p_i \leq \lambda\}}{N-\sum_{i=1}^k m 1\{\widetilde p_i \leq \lambda\}+\kappa} \Big| \epsilon \right)\\
&=& (1-\lambda) E\left( \frac{\sum_{i=1}^k 1\{\widetilde p_i \leq \lambda\}}{\frac{N}{m}-\sum_{i=1}^k  1\{\widetilde p_i \leq \lambda\}+\frac{\kappa}{m}} \Big| \epsilon \right)\\
&\leq& (1-\lambda) E\left( \frac{\sum_{i=1}^k 1\{\widetilde p_i \leq \lambda\}}{k+1-\sum_{i=1}^k  1\{\widetilde p_i \leq \lambda\}} \Big| \epsilon \right)
\leq \lambda
\end{eqnarray*}
since $x \mapsto \frac{x}{N-x+\kappa}$ is increasing for $x\in [0,N+\kappa)$ and  
$m\sum_{i=1}^k 1\{\widetilde p_i \leq \lambda\} \leq km < km+m+N-N_{min} \leq N+\kappa$ holds. 
Moreover, $ \frac{N+\kappa}{m} \geq \frac{N-N_{min} + m + r +n}{m} = \frac{N-N_{min} + m + mk}{m} \geq k+1 $ holds and 
the last inequality follows by the well known result for Binomial variables 
\begin{equation}\label{Theorem004Proof001}
 E\left( \frac{\sum_{i=1}^k 1\{\widetilde p_i \leq \lambda\}}{k+1-\sum_{i=1}^k  1\{\widetilde p_i \leq \lambda\}} \Big| \epsilon \right) 
= \frac{\lambda}{1-\lambda} (1-\lambda^k).
\end{equation}
\hfill $\square$

\par\bigskip\noindent
{\sc Proof of Example \ref{Example6-1}.} 
The proof for the FWER is mostly included in the proof of Theorem \ref{Theorem004}. Notice first that 
\begin{equation}
 \frac{\alpha}{\lambda} \frac{V(\lambda)}{\hat n_0(\kappa)} = \frac{\alpha (1-\lambda)}{\lambda} \frac{X}{k+1-X} \leq 1
\end{equation}
always holds under our assumptions, where $X = \sum_{i=1}^k 1\{\widetilde p_i \leq \lambda\}$ is a binomial 
variable. Thus, we have equality in (\ref{OurCondition001}). By an inspection of the proof above we arrive at 
a sequence of equality with 
\begin{equation}
 FWER = \frac{\alpha(1-\lambda)}{\lambda} E\left( \frac{X}{k+1-X} \right)
\end{equation}
and (\ref{Theorem004Proof001}) can be applied. \hfill $\square$

\section*{Acknowledgements}
The authors are grateful to Helmut Finner and Julia Benditkis for many stimulating discussions and to Veronika Gontscharuk who provided 
a program for the exact computation of the SU FDR for arbitrary critical values in a DU setting, see Figure \ref{fig:Bild1}. Furthermore, 
the authors gratefully acknowledge the helpful suggestions of the referees, who also referred to \cite{guo_sarkar}, the associate editor 
and editor. The authors also wishes to thank the Deutsche Forschungsgemeinschaft, DFG, for financial support.

\end{document}